\documentclass[11pt,letter]{article}
\usepackage{titling}
\usepackage{amsfonts,amssymb,enumitem}
\usepackage{amsmath,amsthm,color}
\usepackage[utf8]{inputenc}
\usepackage[T1]{fontenc}
\usepackage{xspace}
\usepackage{cite}
\usepackage[a4paper, top=2cm, bottom=2cm, left=2.5cm, right=2.5cm, includefoot]{geometry} 
\usepackage{xcolor}
\usepackage{hyperref}
\usepackage{todonotes}
\usepackage{mathtools}
\hypersetup{
	colorlinks=true,
    linkcolor={red!50!black},
    citecolor={blue!50!black},
    bookmarksopen=true,
	  bookmarksnumbered,
	  bookmarksopenlevel=2,
	  bookmarksdepth=3
}
\usepackage{listings}
\usepackage{thmtools,thm-restate}

\usepackage[capitalise]{cleveref}
\crefname{conjecture}{Conjecture}{Conjectures}

\newcommand\extrafootertext[1]{%
    \bgroup
    \renewcommand\thefootnote{\fnsymbol{footnote}}%
    \renewcommand\thempfootnote{\fnsymbol{mpfootnote}}%
    \footnotetext[0]{#1}%
    \egroup
}

\newcommand{\cmsotwo}{$\mathsf{CMSO}_2$\xspace}
\newcommand{\msotwo}{$\mathsf{MSO}_2$\xspace}

\newcommand{\tw}{\mathsf{tw}}
\newcommand{\simw}{\mathsf{simw}}

\newcommand{\ibn}{\mathsf{ibn}}
\newcommand{\tin}{\mathsf{tree} \textnormal{-} \alpha}

\newcommand{\NP}{\textsf{NP}}

\newtheorem{theorem}{Theorem}[section]
\newtheorem{conjecture}[theorem]{Conjecture}

\newtheorem{corollary}[theorem]{Corollary}
\newtheorem{proposition}[theorem]{Proposition}
\newtheorem{lemma}[theorem]{Lemma}
\newtheorem{observation}[theorem]{Observation}

\theoremstyle{definition}

\newtheorem{question}[theorem]{Question}

\usepackage{tikz}
\usepackage{subcaption}

\colorlet{myGreen}{green!50!black}
\colorlet{myLightgreen}{green}
\colorlet{myRed}{red!90!black}
\definecolor{myBlue}{rgb}{0.25, 0.0, 1.0}
\definecolor{myLightBlue}{rgb}{0.39, 0.58, 0.93}
\colorlet{myViolet}{myBlue!55!myRed}
\definecolor{myOrange}{rgb}{1.0, 0.66, 0.07}

\definecolor{CornflowerBlue}{rgb}{0.39, 0.58, 0.93}
\definecolor{DarkGoldenrod}{rgb}{0.72, 0.53, 0.04}
\definecolor{BritishRacingGreen}{rgb}{0.0, 0.26, 0.15}
\definecolor{DarkMagenta}{rgb}{0.55, 0.0, 0.55}
\definecolor{AO}{rgb}{0.0, 0.5, 0.0}
\definecolor{BostonUniversityRed}{rgb}{0.8, 0.0, 0.0}
\definecolor{myRed}{rgb}{0.8, 0.0, 0.0}
\definecolor{DarkMidnightBlue}{rgb}{0.0, 0.2, 0.4}
\definecolor{DarkTangerine}{rgb}{1.0, 0.66, 0.07}
\definecolor{AppleGreen}{rgb}{0.55, 0.71, 0.0}
\definecolor{BrightUbe}{rgb}{0.82, 0.62, 0.91}
\definecolor{Amethyst}{rgb}{0.6, 0.4, 0.8}
\definecolor{DarkGray}{rgb}{0.52, 0.52, 0.51}
\definecolor{Gray}{rgb}{0.66, 0.66, 0.66}
\definecolor{BananaYellow}{rgb}{1.0, 0.88, 0.21}
\definecolor{Amber}{rgb}{1.0, 0.75, 0.0}
\definecolor{LightGray}{rgb}{0.83, 0.83, 0.83}
\definecolor{PrincetonOrange}{rgb}{1.0, 0.56, 0.0}
\definecolor{DeepCarrotOrange}{rgb}{0.91, 0.41, 0.17}
\definecolor{MidnightBlue}{rgb}{0.1, 0.1, 0.44}
\definecolor{HotMagenta}{rgb}{1.0, 0.11, 0.81}
\definecolor{Iceberg}{rgb}{0.44, 0.65, 0.82}
\definecolor{LavenderMagenta}{rgb}{0.93, 0.51, 0.93}
\definecolor{ChromeYellow}{rgb}{1.0, 0.65, 0.0}

\usetikzlibrary{calc}
\usetikzlibrary{fit}
\usetikzlibrary{decorations}
\usetikzlibrary{decorations.pathmorphing}
\usetikzlibrary{decorations.text}
\usetikzlibrary{external}
\usetikzlibrary{shapes,hobby}

\tikzset{
	position/.style args={#1:#2 from #3}{
		at=($(#3)+(#1:#2)$)
	}
}

\tikzset{
  v:main/.style = {draw, circle, scale=0.8, thick,fill=black,inner sep=0.7mm},
  v:mainempty/.style = {draw, circle, scale=0.8, thick,fill=white,inner sep=0.7mm},
  v:middle/.style = {draw, circle, scale=0.3,thick,fill=Gray,color=Gray,inner sep=1mm},
  v:border/.style = {draw, circle, scale=0.75, thick,minimum size=10.5mm},
  v:mainfull/.style = {draw, circle, scale=1, thick,fill},
  v:ghost/.style = {inner sep=0pt,scale=1},
  >={latex},
  e:shiftedright/.style = {decoration={sl, raise=0.65pt},  decorate},
  e:shiftedleft/.style  = {decoration={sl, raise=-0.65pt}, decorate},
  e:marker/.style = {line width=8.5pt,line cap=round,opacity=0.35,color=DarkGoldenrod},
  e:colored/.style = {line width=1.8pt,color=BostonUniversityRed,cap=round,opacity=0.8},
  e:coloredthin/.style = {line width=1.6pt,opacity=1},
  e:coloredborder/.style = {line width=3.4pt},
  e:main/.style = {line width=1pt},
  e:thick/.style = {line width=2pt},
  e:mainthin/.style = {line width=0.6pt},
}

\usepackage{authblk}
\usepackage{microtype}

\title{Treewidth versus clique number. IV. Tree-independence number of graphs excluding an induced star}
\predate{}
\date{}
\postdate{}

\author[1]{Clément Dallard}
\author[2,3]{Matjaž Krnc}
\author[4,5]{O-joung Kwon}
\author[2]{Martin Milanič}
\author[6]{Andrea Munaro}
\author[2,3]{Kenny Štorgel}
\author[5]{Sebastian Wiederrecht}

\newcommand{\email}[1]{%
  \texttt{#1}%
}

\affil[1]{Department of Informatics, University of Fribourg, Switzerland}
\affil[2]{FAMNIT and IAM, University of Primorska, Koper, Slovenia}
\affil[3]{Faculty of Information Studies, Novo mesto, Slovenia}
\affil[4]{Department of Mathematics, Hanyang University, Seoul, South Korea}
\affil[5]{Discrete Mathematics Group, Institute for Basic Science, Daejeon, South Korea}
\affil[6]{Department of Mathematical, Physical and Computer Sciences, University of Parma, Italy}

\begin{document}

\maketitle
\extrafootertext{\scriptsize Emails: 
\email{clement.dallard@unifr.ch},
\email{matjaz.krnc@upr.si},
\email{ojoungkwon@hanyang.ac.kr},
\email{martin.milanic@upr.si}, 
\email{andrea.munaro@unipr.it}, 
\email{kennystorgel.research@gmail.com}, %
\email{sebastian.wiederrecht@gmail.com}.  %
}

\thispagestyle{empty}

\begin{abstract}
\noindent 
Many recent works address the question of characterizing induced obstructions to bounded treewidth.
In 2022, Lozin and Razgon completely answered this question for graph classes defined by finitely many forbidden induced subgraphs.
Their result also implies a characterization of graph classes defined by finitely many forbidden induced subgraphs that are $(\tw,\omega)$-bounded, that is, treewidth can only be large due to the presence of a large clique.
This condition is known to be satisfied for any graph class with bounded tree-independence number, a graph parameter introduced independently by Yolov in 2018 and by Dallard, Milani{\v c}, and {\v S}torgel in 2024.
Dallard et al.\ conjectured that $(\tw,\omega)$-boundedness is actually equivalent to bounded tree-independence number.
We address this conjecture in the context of graph classes defined by finitely many forbidden induced subgraphs and prove it for the case of graph classes excluding an induced star.
We also prove it for subclasses of the class of line graphs, determine the exact values of the tree-independence numbers of line graphs of complete graphs and line graphs of complete bipartite graphs, and characterize the tree-independence number of $P_4$-free graphs, which implies a linear-time algorithm for its computation.
Applying the algorithmic framework provided in a previous paper of the series leads to polynomial-time algorithms for the Maximum Weight Independent Set problem in an infinite family of graph classes.

\medskip
\noindent{\bf Keywords:} tree-independence number, tree decomposition, treewidth, hereditary graph class, line graph, $P_4$-free graph, cograph

\medskip
\noindent{\bf MSC Classes (2020):} 
05C75, %
05C76, %
05C85 %
\end{abstract}

\clearpage
\section{Introduction}

\subsection{Background}

Treewidth is a well-studied and important graph parameter.
Besides playing a crucial role in Graph Minor Theory by Robertson and Seymour (see, e.g.,~\cite{MR2188176}), it is also of significant algorithmic importance.
In particular, Courcelle's theorem~\cite{MR1042649} 
asserts that in any class of graphs with bounded treewidth, any decision problem expressible in {\msotwo} logic can be solved in linear time.\footnote{The result assumes that the graph is equipped with a tree decomposition of bounded width. 
As shown by Bodlaender~\cite{MR1417901}, such a tree decomposition can be computed in linear time.}
An extension of this result to optimization problems was given by Arnborg, Lagergren, and Seese~\cite{MR1105479}.
While these metatheorems are very general with respect to the problem space, their applicability with respect to graph classes is limited to classes of graphs that are sparse, in the sense that they can only have a linear number of edges.
For example, while complete graphs have arguably a very simple structure, they have unbounded treewidth, hence, they are not captured by Courcelle's theorem.

There have been several ways to address this issue in the literature. 
Numerous graph width parameters generalizing treewidth were proposed that can also capture dense graph classes.
This includes clique-width~\cite{MR1743732} (and closely related parameters rank-width~\cite{MR2232389} and Boolean-width~\cite{MR2857670,MR3126918}),
mim-width~\cite{vatshelle2012new},
sim-width~\cite{MR3721445}, and twin-width~\cite{MR4402362}.
Each of these width parameters has some useful algorithmic features, which in some cases include metatheorems for problems expressible in certain logics weaker than {\msotwo} (see~\cite{DBLP:conf/soda/BergougnouxDJ23,MR4402362,MR3126917,MR3126918,MR1739644}).
A different approach, inherently related to treewidth, was recently proposed by Dallard, Milani{\v{c}}, and \v{S}torgel~\cite{MR4334541} who initiated a systematic study of $(\tw,\omega)$-bounded graph classes, that is, graph classes in which the treewidth can only be large due to the presence of a large clique.
More precisely, a graph class is said to be \emph{$(\tw,\omega)$-bounded} if it admits a \emph{$(\tw,\omega)$-binding} function, that is, a function $f$ such that for every graph $G$ in the class and every induced subgraph $H$ of $G$, the treewidth of $H$ is at most $f(\omega(H))$, where $\omega(H)$ denotes the clique number of $H$.
Interestingly, 
this purely structural restriction, even when imposed only on the graphs in the class (and not necessarily for their induced subgraphs), already implies some good algorithmic properties, such as linear-time fixed-parameter tractable algorithms for the \textsc{$k$-Clique} and \textsc{List $k$-Coloring} problems (see
~\cite{MR4332111}), as long as the function $f$ is computable; in some cases, it also leads to improved approximations for the \textsc{Maximum Clique} problem (see~\cite{MR4334541}).

The full extent of algorithmic potential of $(\tw,\omega)$-boundedness is not yet understood.
Yolov~\cite{DBLP:conf/soda/Yolov18} and Dallard et al.\ \cite{dallard2022firstpaper} independently introduced a graph width parameter that, when bounded, implies $(\tw,\omega)$-boundedness as well as polynomial-time solvability of several problems related to independent sets.
This parameter is called \emph{tree-independence number} and is denoted by $\tin(G)$.\footnote{Yolov called it \emph{$\alpha$-treewidth} in~\cite{DBLP:conf/soda/Yolov18}.}
It is defined similarly as treewidth, via tree decompositions, except that the measure of quality of a tree decomposition is changed; instead of measuring the maximum cardinality of a bag, what matters is the \emph{independence number} of the decomposition, defined as the maximum cardinality of an independent set contained in a bag.
For problems related to independent sets, this measure allows for the development of polynomial-time dynamic programming algorithms (see~\cite{DBLP:conf/soda/Yolov18,dallard2022firstpaper}).
If, for a fixed $k$, a graph $G$ with $\tin(G)\le k$ is given without a corresponding tree decomposition, then one can compute in polynomial time a tree decomposition of $G$ with independence number at most $8k$ (see~\cite{dallard2022computing}).
A metatheorem due to Milani{\v{c}} and Rz\k{a}\.zewski~\cite{milanic2022beyond} shows that, for any class of graphs with bounded tree-independence number and any fixed \cmsotwo property, the problem of finding a maximum-weight induced subgraph with bounded chromatic number and satisfying the property is solvable in polynomial time. 
Besides \textsc{Maximum Weight Independent Set}, this framework also captures the problems of computing a maximum-weight induced matching, a maximum-weight induced forest, a maximum-weight planar induced subgraph, and many others.

In summary, boundedness of tree-independence number gives a sufficient condition for \textsc{Maximum Independent Set} and many other problems to be solvable in polynomial time in a $(\tw,\omega)$-bounded graph class.
It is not yet known how restrictive boundedness of tree-independence number is compared to $(\tw,\omega)$-boundedness.
In fact, Dallard et al.\ conjectured that bounded tree-independence number is not only sufficient for $(\tw,\omega)$-boundedness but also necessary.

\begin{conjecture}[Dallard et al.\ \cite{dallard2022secondpaper}]\label{conj:main}
Let $\mathcal{G}$ be a hereditary graph class.
Then $\mathcal{G}$ is $(\tw,\omega)$-bounded if and only if $\mathcal{G}$ has bounded tree-independence number.
\end{conjecture}

While the conjecture is still open, the following partial results are known.
\begin{itemize}
\item The conjecture holds for graph classes closed under the subgraph, topological minor, or minor relation.
Indeed, it follows from Robertson and Seymour's Grid-Minor Theorem (see~\cite{MR0854606}) that in any subgraph-closed graph class, the properties of $(\tw,\omega)$-boundedness, bounded tree-independence number, and bounded treewidth are equivalent (see~\cite[Remark 7.2]{dallard2022secondpaper}).

\item In~\cite{dallard2022secondpaper}, the equivalence between $(\tw,\omega)$-boundedness and bounded tree-independence number was established for graph classes excluding a single graph as an induced subgraph, induced topological minor, or induced minor.

\item Over a series of papers, Abrishami et al.\ recently studied induced obstructions to bounded treewidth (see~\cite{abrishami2022induced1,abrishami2022induced,MR4602718,MR4660624,MR4480493,MR4442958,abrishami2023induced1,abrishami2023induced2}).
In particular, they showed that the class of (even hole, diamond, pyramid)-free graphs is $(\tw,\omega)$-bounded (see~\cite{MR4602718}).
In line with the conjecture, Abrishami et al.\ showed in a follow-up work (see~\cite{abrishami2023tree}) that the class of (even hole, diamond, pyramid)-free graphs has bounded tree-independence number.

\item As shown by Brettell et al.~\cite{brettell2023comparing}, the conjecture holds for classes of bounded mim-width. This follows from the fact that every $(\tw, \omega)$-bounded graph class is $K_{d,d}$-free, for some $d \in \mathbb{N}$, and that $K_{d,d}$-free graphs of bounded mim-width have bounded tree-independence number.
\end{itemize}

\subsection{Our focus}

In~\cite{dallard2022secondpaper} the authors asked whether \cref{conj:main} holds when restricted to graph classes defined by finitely many forbidden induced subgraphs.
This interesting variant of the conjecture can be stated rather explicitly using a result of Lozin and Razgon~\cite{MR4385180} characterizing bounded treewidth within such graph classes.
The result states that a graph class defined by finitely many forbidden induced subgraphs has bounded treewidth if and only if it excludes at least one graph from each of the following four families:
complete graphs, complete bipartite graphs, $\mathcal{S}$, and $L(\mathcal{S})$, where $\mathcal{S}$ is the family of graphs every component of which is a tree with at most three leaves, and $L(\mathcal{S})$ is the family of all line graphs of graphs in $\mathcal{S}$.

This result immediately implies that a graph class defined by finitely many forbidden induced subgraphs is $(\tw,\omega)$-bounded if and only if it excludes at least one graph from each of the following three families:
complete bipartite graphs, $\mathcal{S}$, and $L(\mathcal{S})$.
Thus, when restricted to graph classes defined by finitely many forbidden induced subgraphs, \cref{conj:main} is equivalent to the following.

\begin{conjecture}\label{conjecture:finitely-many-fis}
For any positive integer $d$ and any two graphs $S\in \mathcal{S}$ and $T\in L(\mathcal{S})$, 
the class of $\{K_{d,d},S,T\}$-free graphs has bounded tree-independence number.
\end{conjecture}

The result of Lozin and Razgon implies \cref{conjecture:finitely-many-fis} when restricted to classes of graphs with bounded maximum degree.
Another way to see this is to apply a result of Korhonen~\cite{MR4539481} who proved a conjecture of Aboulker, Adler, Kim, Sintiari, and Trotignon~\cite{MR4286665} (which was also explicitly mentioned in~\cite{MR4660624,MR4480493,MR4602718}) stating that any graph with sufficiently large treewidth and bounded maximum degree contains a large wall or the line graph of a large wall as an induced subgraph. 
 Moreover, a proof of some special cases of \cref{conjecture:finitely-many-fis} can be obtained by combining the aforementioned result of Brettell et al.~\cite[Theorem~6]{brettell2023comparing} with \cite[Theorem~30]{brettell2022mim-width}, which provides several pairs of graphs $S \in \mathcal{S}$ and $T \in L(\mathcal{S})$ such that the mim-width of $\{S, T\}$-free graphs is bounded.

Note that for any positive integer $s$, the $s$-vertex path $P_s$ belongs both to $\mathcal{S}$ as well as to $L(\mathcal{S})$.
The following special case of \cref{conjecture:finitely-many-fis} is already of interest.

\begin{conjecture}\label{conjecture:excluding-a-path}
For any two positive integers $d$ and $s$, the class of $\{K_{d,d},P_s\}$-free graphs has bounded tree-independence number.
\end{conjecture}

Combining known results from the literature implies the validity of \cref{conjecture:excluding-a-path} for the case $s = 4$, with a bound that is exponential in $d$.
This is a consequence of the following four results from the literature: first, every $P_4$-free graph is distance-hereditary (see Bandelt and Mulder~\cite{MR0859310}); second, the rank-width of distance-hereditary graphs is at most $1$ (see Oum~\cite{MR2156341}); third, the mim-width of a graph is at most its rank-width (see Vatshelle~\cite{vatshelle2012new}); and fourth, the tree-independence number of a $K_{d,d}$-free graph $G$ with mim-width less than $k$ is less than $6(2^{d+k-1}+dk^{d+1})$ (see Brettell et al.~\cite[Theorem~6]{brettell2023comparing}).
Hence, if $d$ is a positive integer and $G$ is a $\{K_{d,d}, P_4\}$-free graph, then $\tin(G)<6(d+1)2^{d+1}$.

\subsection{Our results}

The aim of this paper is to provide further partial support for \cref{conj:main,conjecture:finitely-many-fis,conjecture:excluding-a-path}.
In particular, we consider \cref{conjecture:finitely-many-fis} and show that the assumption of bounded maximum degree can be relaxed to excluding a fixed induced star, that is, a complete bipartite graph $K_{1,d}$, where $d$ is a fixed positive integer.
In other words, we prove the following weakening of \cref{conjecture:finitely-many-fis}.

\begin{restatable}{theorem}{restateMain}\label{thm:main}
For any positive integer $d$ and any two graphs $S\in \mathcal{S}$ and $T\in L(\mathcal{S})$, 
the class of $\{K_{1,d},S,T\}$-free graphs has bounded tree-independence number.
\end{restatable}

\cref{thm:main} shows that \cref{conj:main} holds for graph classes defined by a set $\mathcal{F}$ of finitely many forbidden induced subgraphs, as long as $\mathcal{F}$ contains some star.
Our constructive proof gives an explicit upper bound on the tree-independence number that is polynomial in $d$, $|V(S)|$, and $|V(T)|$ (in fact, linear in each of $|V(S)|$ and $|V(T)|$; we refer to \cref{cor:main} for the exact statement). 

The first step in our proof of \cref{thm:main} is the following weakening of \cref{conjecture:excluding-a-path}, which we state with an explicit bound.

\begin{restatable}{theorem}{excludePathOriginal}\label{thm:excludepathoriginal}
Let $d\ge 2$ and $s\ge 3$ be integers and let $G$ be a $\{K_{1,d},P_s\}$-free graph.
Then \[\tin(G)\le (d-1)(s-2)\,.\]
\end{restatable}

\cref{thm:excludepathoriginal} can be derived from known results in the literature regarding tree decompositions of graphs excluding all sufficiently long induced cycles.
Bodlaender and Thilikos proved in~\cite{MR1478240} that in the absence of long induced cycles, bounded degree implies bounded treewidth.
They gave an upper bound of $\Delta(\Delta-1)^{s-3}$ on the treewidth of any graph with maximum degree at most $\Delta$ and no induced cycles of length more than $s$, where $s\ge 3$ is a fixed constant.
This bound was improved to $\mathcal{O}(\Delta s)$ by Kosowski, Li, Nisse, and Suchan~\cite{MR3355834}, who gave an $\mathcal{O}(|E(G)|^2)$ algorithm that computes, given a graph $G$ with no induced cycles of length more than $s$, a tree decomposition of $G$ in which the subgraph induced by each bag has a dominating path with at most $s-1$ vertices.
This result was further improved by Seymour~\cite{MR3425243}, who showed the existence of a tree decomposition in which the subgraph induced by each bag has a dominating path with at most $s-2$ vertices.
Since in a $K_{1,d}$-free graph, the closed neighborhood of any vertex induces a subgraph with independence number at most $d-1$,
this implies \cref{thm:excludepathoriginal}.

We give an alternative short proof of \cref{thm:excludepathoriginal} yielding the same bound, by adapting Gyárfás's proof of the result that any class of graphs excluding a fixed path as an induced subgraph is $\chi$-bounded~\cite{MR951359}.

For general graph classes excluding an induced star, \cref{conj:main} remains open, even in the case of graph classes excluding the claw (the star $K_{1,3}$).
As our next result, we show that \cref{conj:main} holds for subclasses of the class of line graphs (which are all known to exclude the claw).

\begin{restatable}{theorem}{propLineGraphs}\label{prop:line-graphs}
Let $\mathcal{G}$ be a class of graphs and let $L(\mathcal{G})$ be the class of line graphs of graphs in~$\mathcal{G}$. 
Then, the following statements are equivalent.
\begin{enumerate}
\item The class $L(\mathcal{G})$ is $(\tw,\omega)$-bounded.
\item The class $L(\mathcal{G})$ has bounded tree-independence number.
\item The class $\mathcal{G}$ has bounded treewidth.
\end{enumerate}
\end{restatable}

\cref{prop:line-graphs} complements the result of Brettell et al.~\cite{brettell2023comparing} showing that on any class of line graphs $L(\mathcal{G})$, tree-independence number, clique-width, mim-width, and sim-width are in fact all equivalent, in the sense that each of these parameters is bounded if and only if all the others are.
Indeed, we prove a stronger version of \cref{prop:line-graphs} as follows.
For a graph class $\mathcal{G}$, we denote by $I(\mathcal{G})$ the class of all intersection graphs of collections of connected subgraphs of some member of~$\mathcal{G}$.
We show that the conclusion of \cref{prop:line-graphs} holds for $I(\mathcal{G})$ for \textit{any} graph class $\mathcal{G}$.

We also determine the exact values of the tree-independence number of line graphs of complete graphs and line graphs of complete bipartite graphs.
These results complement similar results for treewidth and sim-width (denoted by $\tw$ and $\simw$, respectively; see~\cite{Luc07,MR3820302,HW15,brettell2023comparing}) and are interesting in view of the inequalities $\simw(G)\le \tin(G)\le \tw(G)+1$ valid for any graph $G$.
For the first inequality, see~\cite[Lemma 5]{MR4563598} or~\cite[Theorem~2 or Theorem~15]{DBLP:conf/wg/BergougnouxKR23}; the second one follows immediately from the definitions.

Finally, regarding \cref{conjecture:excluding-a-path} ``in the other direction'', keeping the assumption that a complete bipartite graph is excluded, but limiting the length of the excluded path, we improve the known result about the first nontrivial case of the conjecture in this regard, that is, the case $s = 4$.
For this case we improve the aforementioned exponential upper bound for the tree-independence number to the following sharp upper bound.

\begin{restatable}{theorem}{restateKdd}\label{tin-of-P4-tree-graphs-original}
Let $G$ be a $\{K_{d,d},P_4\}$-free graph, where $d\ge 2$ is an integer. Then $\tin(G)\le d-1$.
\end{restatable}

The proof of \cref{tin-of-P4-tree-graphs-original} is constructive and leads to a linear-time algorithm for determining the tree-independence number of a given $P_4$-free graph.

\subsection{Algorithmic implications}

\cref{thm:main} together with known results on tree-independence number (see~\cite{DBLP:conf/soda/Yolov18,dallard2022computing,dallard2022firstpaper}) imply that all the good algorithmic properties of graphs with bounded tree-independence number (see~\cite{DBLP:conf/soda/Yolov18,dallard2022firstpaper,milanic2022beyond}) hold for any class of graphs excluding a star, a graph from $\mathcal{S}$, and a line graph of a graph from $\mathcal{S}$; in particular, on such graph classes a number of \NP-hard problems can be solved in polynomial time.
We only state and explicitly discuss the corresponding result for the most well-known and studied of these problems,  namely \textsc{Maximum Weight Independent Set}: Given a graph $G$ and a vertex weight function $w:V(G)\to \mathbb{Q}_+$, compute an independent set $I$ in $G$ maximizing its weight $\sum_{x\in I}w(x)$.

\begin{corollary}\label{cor:main-algorithm}
For any positive integer $d$ and any two graphs $S\in \mathcal{S}$ and $T\in L(\mathcal{S})$, \textsc{Maximum Weight Independent Set} is solvable in polynomial time in the class of $\{K_{1,d},S,T\}$-free graphs.
\end{corollary}

This result is interesting in view of the state-of-the-art regarding the complexity of the problem in classes of graphs excluding finitely many forbidden induced subgraphs.
For any finite set $\mathcal{F}$ of graphs, Alekseev proved in~\cite{MR765704} that the \textsc{Maximum Independent Set} is \NP-hard in the class of $\mathcal{F}$-free graphs, unless $\mathcal{F}$ contains a member of $\mathcal{S}$.
On the other hand, Lozin conjectured that the problem is solvable in polynomial time if $\mathcal{F}$ contains a member of $\mathcal{S}$ (see~\cite{MR3695266}).
While the conjecture is still widely open, even in the case of excluding a path (see, e.g.,~\cite{DBLP:conf/focs/GartlandL20,DBLP:conf/sosa/PilipczukPR21,MR4374260}), although a subexponential algorithm was given by Majewski et al.\ \cite{DBLP:conf/icalp/MajewskiM0OPRS22}.
Furthermore, the conjecture was confirmed for the case of graphs with bounded maximum degree by Abrishami, Chudnovsky, Dibek, and Rz\k{a}\.{z}ewski (see~\cite{DBLP:conf/soda/AbrishamiCDR22}).
A shorter proof was given recently by Abrishami, Chudnovsky,  Pilipczuk, and Rz\k{a}\.zewski~\cite{abrishami2023max}, who also generalized the bounded degree assumption to the assumption that some complete bipartite graph is excluded as a subgraph.
\cref{cor:main-algorithm} shows that, if in addition to a graph from $\mathcal{S}$ the line graph of such a graph is also excluded, then the maximum degree assumption can be relaxed to the assumption of excluding an induced star.

Let us remark that in the special case of excluding a \emph{claw}, that is, the graph $K_{1,3}$, the \textsc{Maximum Weight Independent Set} problem is known to be solvable in polynomial time without any additional assumptions (see~\cite{MR0579076,MR1845110,zbMATH01859168}, as well as~\cite{MR2463423,MR3763297} for generalizations).
On the other hand, for $d>3$, the \textsc{Maximum Weight Independent Set} is \NP-hard in the class of $K_{1,d}$-free graphs, but admits a polynomial-time $(d/2)$-approximation algorithm (see~\cite{MR1810314,DBLP:conf/stacs/Neuwohner21}).

\subsection{Structure of the paper}

\cref{sec:prelims} introduces the notations and observations that  we use throughout the paper.
In \cref{sec:star-path} we discuss the special case of \cref{conj:main} restricted to $\{K_{1,d},P_s\}$-free graphs.
\cref{sec:main} deals with our main technical result, where we settle the case of $\{K_{1,d},S,T\}$-free graphs, for any $S\in \mathcal{S}$ and $T\in L(\mathcal{S})$, which proves \cref{thm:main}.
In \cref{sec:line} we establish the validity of \cref{conj:main} for subclasses of the class of line graphs and determine the exact values of tree-independence number of line graphs of complete graphs and complete bipartite graphs.
In \cref{sec:p4-free} we develop a linear-time algorithm for computing the tree-independence number of a $P_4$-free graph.
We conclude the paper with some open questions and insights for future research in \cref{sec:conc}.

\section{Preliminaries\label{sec:prelims}}

Given two integers $a,b\in\mathbb{Z}$, we denote by $[a,b]$ the set $\{ x\in\mathbb{Z} \mid a\leq x\leq b \}$.
Notice that $[a,b]$ is empty in the case $b<a$.
Moreover, for a single integer $k\in\mathbb{Z}$, we denote by $[k]$ the set $[1,k]$.

Concerning graph notation we follow mostly the 
conventions from \cite{Diestel:GraphTheory5}.
The graphs considered in this paper are finite and simple, that is, they do not contain loops or parallel edges.
For a graph $G$ and a set $S \subseteq V(G)$, we write $G-S$ to denote the graph obtained from $G$ by deleting the vertices in $S$ (and the edges incident to those vertices).
 In the case when $S$ consists of a single vertex $v$, we write $G-v$ for $G-S$.
A~graph class is \emph{hereditary} if it is closed under vertex deletion.
We denote by $G_1+G_2$ the \emph{disjoint union} of two graphs $G_1$ and $G_2$, and by $G_1\ast G_2$ their \emph{join}, that is, the graph obtained from the disjoint union of $G_1$ and $G_2$ by adding all edges joining a vertex of $G_1$ with a vertex of $G_2$.

For a graph $G$ the operation of \emph{subdividing} the edge $e=uv\in E(G)$ is the deletion of $e$ and the introduction of a vertex $w$ adjacent exactly to the vertices $u$ and $v$.
A graph $G$ is a \emph{subdivision} of a graph $H$ if it can be obtained from $H$ by a sequence of edge subdivisions.

A \emph{path} is a graph $P$ with vertex set $\{v_1,\dots,v_{\ell}\}$ such that $v_iv_{i+1}$ is an edge for every $i\in[\ell-1]$, and there are no other edges.
We say that $v_1$ and $v_{\ell}$ are the \emph{endpoints} of $P$, all other vertices of $P$ are \emph{internal}, and the \emph{length} of $P$ is $\ell-1$.
For a graph $G$ and vertex sets $X,Y\subseteq V(G)$, an \emph{$(X, Y)$-path} is a path in $G$ with one endpoint in $X$, the other endpoint in $Y$, and no internal vertex in $X\cup Y$.
In the case any of $X$ and $Y$ consists of a single vertex, we may write the vertex instead of the set.
Given a path $P$ and two vertices $u,v\in V(P)$ we denote by $uPv$ the unique $(u, v)$-path in $P$ (writing also $uv$ in the case $uPv$ has length one).
Moreover, given two paths $P$ and $Q$, a vertex $x\in V(P)$, a vertex $y\in V(P)\cap V(Q)$, and a vertex $z\in V(Q)$, we write $xPyQz$ for the union of the paths $xPy$ and $yQz$.

An \emph{independent set} in a graph $G$ is a set of pairwise non-adjacent vertices, and a \emph{clique} in $G$ is a set of pairwise adjacent vertices. The \emph{independence number} of a graph $G$, denoted by $\alpha(G)$, is the maximum size of an independent set in $G$. The \emph{clique number} of a graph $G$, denoted by $\omega(G)$, is the maximum size of a clique in $G$.
Given two positive integers $a$ and $b$, the \emph{complete bipartite graph} $K_{a,b}$ is a graph whose vertex set admits a partition into two independent sets $A$ and $B$ such that $|A| = a$, $|B| = b$, and every vertex in $A$ is adjacent to every vertex in $B$.

Given a graph $G$, a \emph{tree decomposition} of $G$ is a pair $\mathcal{T}=(T,\beta)$ of a tree $T$ and a function $\beta\colon V(T)\to 2^{V(G)}$ whose images are called the \emph{bags} of $\mathcal{T}$ such that every vertex belongs to some bag, for every $e\in E(G)$ there exists some $t\in V(T)$ with $e\subseteq \beta(t)$, and for every vertex $v\in V(G)$ the set $\{t\in V(T) \mid v\in \beta(t)\}$ induces a subtree of $T$.
We refer to the vertices of $T$ as the \emph{nodes} of the tree decomposition $\mathcal{T}$. 
If $T$ is a path, then we call $\mathcal{T}$ a \emph{path decomposition} of $G$.
The \emph{width} of $\mathcal{T}$ equals $\max_{t \in V(T)} |\beta(t)|-1$, and the \emph{treewidth} of a graph $G$, denoted by $\tw(G)$, is the minimum possible width of a tree decomposition of $G$. 
The \emph{independence number} of $\mathcal{T}$, denoted by $\alpha(\mathcal{T})$, is defined as  \[\alpha(\mathcal{T})=\max_{t\in V(T)} \alpha(G[\beta(t)]).\]
The \emph{tree-independence number} of a graph $G$, denoted by $\tin(G)$, is the minimum independence number among all possible tree decompositions of $G$.
Observe that every graph $G$ satisfies $\tin(G) \leq \alpha(G)$. 
The tree-independence number of a graph is bounded from below by its sim-width, a parameter introduced in 2017 by Kang, Kwon, Str\o{}mme, and Telle~\cite{MR3721445}.
Since we will not need the precise definition of sim-width in this paper, we refer the reader to~\cite{MR3721445} for the definition.

The following monotonicity of treewidth is well known.
Given two graphs $G$ and $H$, we say that $H$ is a \emph{minor} of $G$ if $H$ can be obtained from a subgraph of $G$ by a sequence of edge contractions.

\begin{proposition}[folklore]\label{lem:tw-minor}
Let $G$ be a graph and $H$ a minor of $G$.
Then $\tw(H)\le \tw(G)$.
\end{proposition}

A graph is said to be \emph{chordal} if it does not contain any induced cycles of length at least four.
Treewidth can be defined in many equivalent ways. 
One of the characterizations is as follows (see, e.g.,~\cite{MR1647486}).

\begin{theorem}\label{thm:tw-via-chordal-graphs}
Let $G$ be a graph. 
Then, the treewidth of $G$ equals
the minimum value of $\omega(G')-1$ such that $G$ is a subgraph of $G'$ and $G'$ is chordal. 
\end{theorem}

We will need the following results on tree decompositions and tree-independence number from Dallard et al.~\cite{dallard2022firstpaper}.

\begin{lemma}\label{lem:closednbr}
Let $G$ be a graph and let $\mathcal{T}=(T,\beta)$ be a tree decomposition of $G$. 
Then there exists a vertex $v\in V(G)$ and a node $t\in V(T)$ such that $N[v]\subseteq \beta(t)$.
\end{lemma}

Given two graphs $G$ and $H$, we say that $H$ is an \emph{induced minor} of $G$ if $H$ can be obtained from $G$ by a sequence of vertex deletions and edge contractions.

\begin{proposition}\label{lem:tree-independence number induced-minor}
Let $G$ be a graph and $H$ an induced minor of $G$.
Then $\tin(H)\le \tin(G)$.
\end{proposition}

\begin{proposition}\label{tin-of-Knn}
For every positive integer $n$, we have $\tin(K_{n,n})  = n$.
\end{proposition}

\section{Tree-independence number of \texorpdfstring{$\{K_{1,d},P_s\}$}{\{K\_\{1,d\},Pₛ\}}-free graphs\label{sec:star-path}}

This section contains the following important preliminary result.
\excludePathOriginal*
We will utilize this result in the proof of our main theorem in the next section.

\begin{lemma}\label{lem:an-easy-lemma}
For every graph $G$ and a set $S\subseteq V(G)$,
we have 
\[\tin(G)\le \tin(G-S)+\alpha(G[S])\,.\]
\end{lemma}

\begin{proof}
Let $G' = G-S$ and $\mathcal{T}' = (T,\beta')$ be a tree decomposition of $G'$ with minimum independence number.
We construct a tree decomposition $\mathcal T = (T,\beta)$ of $G$ from $\mathcal{T}'$ by setting $\beta(t) \coloneqq \beta'(t) \cup S$ for every $t \in V(T)$.
Clearly, for every bag $\beta(t)$, $t\in V(T)$, we have $\alpha(G[\beta(t)]) \leq \alpha(G[\beta'(t)]) + \alpha(G[S])$.
Hence, we obtain that $\tin(G) \leq \alpha(\mathcal{T}) \leq \alpha(\mathcal{T}') + \alpha(G[S]) = \tin(G-S) + \alpha(G[S])$, as claimed.
\end{proof}

We show \cref{thm:excludepathoriginal} by adapting Gyárfás's proof of $\chi$-boundedness of any class of graphs excluding a fixed path as an induced subgraph~ \cite{MR951359}.

\begin{proof}[Proof of \cref{thm:excludepathoriginal}]
Let $G$ be a $K_{1,d}$-free graph.
Assume that $\tin(G)\ge (d-1)(s-2)+1$.
We show that $G$ contains an induced $P_s$.
To this end, we construct a sequence of connected induced subgraphs $G_1,\ldots, G_s$ of $G$ and an induced $P_s$ in $G$ with vertex set $\{v_1,\ldots, v_s\}$ such that for all $i\in [s]$, the following properties hold:
\begin{enumerate}[label=(\roman*)]
\item\label{item1} $v_i\in V(G_i)$ and if $i<s$, then $v_i$ has a neighbor in $G_i$. 
\item\label{item2} For all $j\in [i-1]$ and all $v\in V(G_i)$, the vertices $v$ and $v_j$ are adjacent in $G$ if and only if $j = i-1$ and $v = v_i$.
\item\label{item3} $\tin(G_i)\ge (d-1)(s-i-1)+1$.
\end{enumerate}
Note that property~\ref{item2} for $i = s$ implies that the subgraph of $G$ induced by $\{v_1,\ldots, v_s\}$ is indeed isomorphic to $P_s$. 

Let $G_1$ be a connected component of $G$ such that $\tin(G_1)\ge (d-1)(s-2)+1$ and let $v_1$ be an arbitrary vertex in $G_1$.
Note that $v_1$ has a neighbor in $G_1$ since 
$\tin(G_1)\ge s-1\ge 2$.

Suppose that $i\in [s-1]$ and that we have already defined the graphs $G_1,\ldots, G_i$ and the vertices $v_1,\ldots, v_i$ such that the properties~\ref{item1}--\ref{item3} hold.
We show how to define $G_{i+1}$ and $v_{i+1}$.
We consider two cases depending on the value of $i$.

Consider first the case when $i\le s-2$.
Let $A$ be the set of vertices of $G_i$ adjacent to $v_i$ and let $B = V(G_i)\setminus(A\cup \{v_i\})$.
Since $G$ is $K_{1,d}$-free and $G_i$ is an induced subgraph of $G$, the subgraph of $G_i$ induced by $A\cup \{v_i\}$ has independence number at most $d-1$.
\cref{lem:an-easy-lemma} implies that $\tin(G_i)\le \tin(G_i[B]) + d-1$ and hence
\begin{align*}
\tin(G_i[B]) &\ge \tin(G_i)-(d-1)\\
&\ge (d-1)(s-i-1)+1-(d-1)\\
&\ge (d-1)(s-i-2)+1\,.
\end{align*}
Note that $(d-1)(s-i-2)+1\ge 1$ since $i\le s-2$.
Thus, $B\neq \emptyset$ and there exists a connected component $H$ of the subgraph of $G_i$ induced by $B$ such that
\begin{align}
\tin(H)\ge (d-1)(s-i-2)+1\,.\label{eq:tin_H}
\end{align}
Since $G_i$ is connected, there exists a vertex in $A$ having a neighbor in $H$. 
We define $v_{i+1}$ to be any such vertex and $G_{i+1}$ to be the subgraph of $G_i$ induced by $V(H)\cup \{v_{i+1}\}$.
By construction, the graph $G_{i+1}$ is a connected induced subgraph of $G$.
Let us verify that the properties~\ref{item1}--\ref{item3} hold for $i+1$.
For property~\ref{item1}, we have $v_{i+1}\in V(G_{i+1})$ and $v_{i+1}$ has a neighbor in $V(H)\subseteq V(G_{i+1})$.
For property~\ref{item2}, consider an arbitrary $j\in \{1,\ldots, i\}$ and a vertex $v\in V(G_{i+1})$.
By the definition of $v_{i+1}$, the vertices $v_i$ and $v_{i+1}$ are adjacent in $G$.
Furthermore, since $v\in V(G_{i+1}) = V(H)\cup\{v_{i+1}\}\subseteq A\cup B\subseteq V(G_i)$, property~\ref{item2} for $i$ implies that $v$ is not adjacent to $v_j$ if $j<i$.
Moreover, if $v\in V(G_{i+1})\setminus \{v_{i+1}\}$, then $v\in V(H)\subseteq B$ and consequently $v$ is not adjacent to $v_i$ in $G$.
This establishes property~\ref{item2} for $i+1$.
Property~\ref{item3} for $i+1$ follows from the fact that $H$ is an induced subgraph of $G_{i+1}$ and hence we have $\tin(G_{i+1})\ge \tin(H)\ge (s-i-2)(d-1)+1$, where the second inequality follows from \cref{eq:tin_H}.

To complete the proof, consider the case when $i = s-1$.
By property~\ref{item1}, vertex $v_{s-1}$ has a neighbor in $G_{s-1}$.
Let $v_{s}$ be any such neighbor and let $G_{s}$ be the one-vertex subgraph of $G$ induced by $v_s$.
We need to verify properties~\ref{item1}--\ref{item3} for $i = s$.
Property~\ref{item1} holds trivially.
For property~\ref{item2}, consider an arbitrary $j\in [s-1]$ and a vertex $v\in V(G_{s})$.
Then $v = v_s$.
By the definition of $v_{s}$, the vertices $v_{s-1}$ and $v_{s}$ are adjacent in $G$.
Furthermore, property~\ref{item2} for $i = s-1$ implies that $v_s$ is not adjacent to $v_j$ if $j<s-1$.
This establishes property~\ref{item2} for $i = s$.
Property~\ref{item3} for $i = s$ simplifies to
$\tin(G_{s})\ge 2-d$, which is clearly true.
\end{proof}

\section{Tree-independence number of \texorpdfstring{$\{K_{1,d},S,T\}$}{\{K\_\{1,d\},S,T\}}-free graphs\label{sec:main}}

Recall that we denote by $\mathcal{S}$ the family of graphs every component of which is a tree with at most three leaves, and by $L(\mathcal{S})$ the family of all line graphs of graphs in $\mathcal{S}$.
In this section we prove our main result, \cref{thm:main}, which we restate here for convenience.

\restateMain*

We first consider the case when $S$ and $T$ are connected.
For $p,q,r\ge 1$, let $S_{p,q,r}$ be the graph obtained from the claw by subdividing one edge $p-1$ times, another $q-1$ times, and the last one $r-1$ times.
By $T_{p,q,r}$ we denote the line graph of $S_{p,q,r}$.
Note that any connected graph $S\in \mathcal{S}$ is either a path or isomorphic to $S_{p,q,r}$ for some $p,q,r\ge 1$.
Similarly, any connected graph $T\in L(\mathcal{S})$ is either a path or isomorphic to $T_{p,q,r}$ for some $p,q,r\ge 1$.
For convenience we will write $T_p$ and $S_p$ as shorthands for $T_{p,p,p}$ and $S_{p,p,p}$, respectively.

In the following we aim to prove that any $K_{1,d}$-free graph $G$ which also excludes both $S_p$ and $T_p$ as induced subgraphs has bounded tree independence number.
Our approach is inspired by Lozin and Rautenbach~\cite{MR2112498}.

The core of our proof of \cref{thm:main} is to show it for the case where both $S$ and $T$ are connected.
We sketch the proof for this case.
We may assume that $G$ contains a long induced path $P$ since, otherwise, \cref{thm:excludepathoriginal} would yield a bound on the tree-independence number immediately. 
Note that here we do not necessarily take a longest induced path because we also want to bound the independence number of $G[N[V(P)]]$.

We next show that no component $D$ of $G-N[V(P)]$ contains a long induced cycle. 
This is because if $D$ has a long induced cycle, then by taking a shortest path from $P$ in $G$, we can find an induced $S_p$ or $T_p$. 
Now, if each component $D$ has bounded tree-independence number, then we can merge 
tree decompositions of components by adding a new bag consisting of $N[V(P)]$ and adding $N[V(P)]$ to all bags of previous decompositions. 
This will show that $G$ has bounded tree-independence number. 
Therefore, we may assume by \cref{thm:excludepathoriginal} that there is a component $D$ of $G-N[V(P)]$ having a long induced path. 
We take a longest induced path $Q$ in $D$.

Then, we prove two main lemmas.
First, we show that $D[N_{D}[V(Q)]]$ admits a path decomposition of bounded independence number.
Second, we show that for every component $H$ of $D-N_{D}[V(Q)]$, there is a bag of the path decomposition containing all the neighbors of $H$ in $D$.
The absence of long induced cycles in $D$ is used to show this second lemma.
Because of the maximality of $Q$, we can show that no component $H$ can have a long induced path, and thus it has bounded tree-independence number.
Using the two lemmas, we finally derive that $D$ has bounded tree-independence number.
This completes the proof sketch for the case when $S$ and $T$ are connected.

\medskip

Let $G$ be a graph, $P$ be an induced path in $G$, and $v\in V(G)\setminus V(P)$ be a vertex with at least one neighbor on $P$.
A \emph{segment} of $P$ with respect to $v$ is a maximal subpath $P'$ of $P$ whose interior is disjoint from $N(v)$.
Note that the edge set of $P$ is partitioned into the edge sets of its segments with respect to $v$.
Similarly, for an induced cycle $C$ in $G$ and $v\in V(G)\setminus V(C)$ having at least two neighbors in $C$, a \emph{segment} of $C$ with respect to $v$ is a maximal subpath $P'$ of $C$ whose interior is disjoint from $N(v)$.

\begin{lemma}\label{lemma:fewneighbors}
Let $d$ be a positive integer, let $G$ be a $K_{1,d}$-free graph, let $P$ be an induced path in $G$, and $v\in V(G)\setminus V(P)$.
Then $v$ has at most $2(d-1)$ neighbors on $P$.
\end{lemma}

\begin{proof}
Let $X\subseteq V(P)$ be any set of vertices of size at least $2d-1$.
As $P$ is an induced bipartite subgraph of $G$, at least $d$ vertices from $X$ belong to the same color class of $P$ and therefore form an independent set in $G$.
Hence $N(v)\cap V(P)$ cannot be larger than $2(d-1)$.
\end{proof}

\begin{lemma}\label{lem:longsegment}
Let $d$ and $p$ be positive integers, let $G$ be a $K_{1,d}$-free graph, let $P$ be an induced path in $G$, and $v\in V(G)\setminus V(P)$. If $P$ has at least $dp$ vertices, then there is a segment of $P$ with respect to $v$ that has at least $p-1$ vertices that are not adjacent to $v$.
\end{lemma}

\begin{proof}
Let $P=v_1\dots v_n$ where $n\geq dp$. 
For each $i\in [d]$, let $Q_i=v_{(i-1)p+1}Pv_{ip-1}$. 
Observe that there are no edges between $Q_{i_1}$ and $Q_{i_2}$ for distinct $i_1, i_2\in [d]$. 
If $v$ has a neighbor in $Q_i$ for each $i\in [d]$, then $G$ contains $K_{1,d}$ as an induced subgraph. 
Therefore, there is a $j\in [d]$ such that $v$ has no neighbor in $Q_j$. 
Then, the segment of $P$ containing $Q_j$ has at least  $p-1$ vertices that are not adjacent to $v$.  
\end{proof}

We next show that a connected $\{K_{1,d}, S_p, T_p\}$-free graph cannot have an induced subgraph that is a disjoint union of a long path and a long cycle.

\begin{lemma}\label{lemma:apathskillsallcycles}
Let $d$ and $p$ be positive integers, let $G$ be a connected $\{K_{1,d},S_p,T_p\}$-free graph, and let $P$ be an induced path in $G$ on at least $dp$ vertices.
Then $G-N[V(P)]$ does not contain an induced cycle of length at least $d(2p+2)$.
\end{lemma}

\begin{proof}
Towards a contradiction, assume that there exists an induced cycle $C=w_0w_1\dots w_{\ell-1}w_0$ in $G-N[V(P)]$ where $\ell\geq d(2p+2)$.
Let $Q=z_0z_1\dots z_r$ be a shortest $(V(P),V(C))$-path in $G$; in particular, $z_0\in V(P)$ and $z_r\in V(C)$.
Notice that $z_1\notin V(P)\cup V(C)$.
Observe that $C-z_r$ is an induced path on $\ell-1$ vertices.

Since $P$ has at least $dp$ vertices, by \cref{lem:longsegment}, there is a segment $P'$ of $P$ with respect to $z_1$ that has at least $p-1$ vertices that are not adjacent to $z_1$.

Assume first that $z_{r-1}$ has a unique neighbor on $C$, that is, $N(z_{r-1})\cap V(C) =\{z_r\}$.
In this case, let $j\in [\ell]$ such that $z_r = w_j$ and let $R$ be the subpath of $C$ of length $2p$ with $R=w_{j-p}w_{j-p+1}\dots w_{j-1}w_jw_{j+1}\dots w_{j+p}$ (indices modulo $\ell$).
Then, observe that the graph $G[V(P')\cup (V(Q)\setminus\{z_0\})\cup V(R)]$ contains $S_p$ as an induced subgraph, which is a contradiction.

Assume now that $z_{r-1}$ has at least two neighbors on $C$.
For each $i\in [d]$, let $D_i$ be the subpath of $C$ from $w_{(i-1)(2p+2)+1}$ to $w_{i(2p+2)-1}$ that does not contain $w_{i(2p+2)}$  (indices modulo $\ell$). 
Observe that there are no edges between $D_{i_1}$ and $D_{i_2}$ for distinct $i_1, i_2\in [d]$. 
If $z_{r-1}$ has a neighbor in each $D_i$, then $G$ contains $K_{1,d}$ as an induced subgraph. Therefore, there is $j\in [d]$ such that $z_{r-1}$ has no neighbor in $D_j$.
This implies that there is a segment $R$ of $C$ with respect to $z_{r-1}$ that has at least $2p+1$ vertices that are not adjacent to $z_{r-1}$.
Let $R=w_jw_{j+1}\dots w_{j+s}$ (indices modulo $\ell$) be such a segment.
Note that $N(z_{r-1})\cap V(R)=\{w_j,w_{j+s}\}$.
If $w_j$ and $w_{j+s}$ are adjacent in $G$ (that is, if $s = \ell-1$), then the graph $G[V(P')\cup (V(Q)\setminus\{z_0,z_r\})\cup V(R)]$ contains $T_p$ as an induced subgraph, which is a contradiction.
Therefore, $w_j$ and $w_{j+s}$ are nonadjacent.
But then $G[V(P')\cup (V(Q)\setminus\{z_0,z_r\})\cup V(R)]$ contains $S_p$ as an induced subgraph, again a contradiction.
\end{proof}

We next show that, whenever we have a bound on the number of vertices of an induced path $P$ in a $K_{1,d}$-free graph, we also obtain a bound on the independence number of the closed neighborhood of $P$.

\begin{lemma}\label{lemma:boundedalphapath}
Let $d\ge 1$ and $q\geq 2$ be integers, $G$ be a connected $K_{1,d}$-free graph, and let $P$ be an induced path in $G$ on $q$ vertices.
Then $\alpha(G[N[V(P)]])\leq (d-1)q$.
\end{lemma}

\begin{proof}
First notice that we may assume $d\geq 2$ since otherwise no $P_2$ could exist in $G$.
Let $I$ be a maximum independent set of $G[N[V(P)]]$.
Then, for every $v\in V(P)$, $I$ either contains $v$ or at most $d-1$ vertices of $N(v)$.
As $P$ has $q$ vertices, the claim follows immediately.
\end{proof}

We have seen in \cref{lemma:apathskillsallcycles} that removing the closed neighborhood of an induced path on a specific number of vertices (which is a set of small independence number by \cref{lemma:boundedalphapath}) leaves a graph without long induced cycles.
This implies a powerful separation property for the closed neighborhoods of short subpaths of long induced paths within the remaining graph.
This is expressed in the following two lemmas.

\begin{lemma}\label{lemma:pathinterval1}
Let $d\ge 1$ and $q\ge 3$ be integers, let $G$ be a connected $K_{1,d}$-free graph without induced cycles of length at least $q$, let $P$ be an induced path in $G$, and let $v\in V(G)\setminus V(P)$.
Then there exists a path $P_v\subseteq P$ with at most $2(d-1)(q-2)$ vertices such that $N(v)\cap V(P)\subseteq V(P_v)$ and each endpoint of $P_v$ is adjacent to $v$.
\end{lemma}

\begin{proof}
Notice that if $d=1$, then $G$ is edgeless and thus we can select $P_v$ to be the empty path.
Hence, we assume that $d\geq 2$.
By \cref{lemma:fewneighbors}, $v$ has at most $2(d-1)$ neighbors on $P$.
Let $P_v$ be the shortest subpath of $P$ containing all neighbors of $v$ on $P$. By the minimality of $P_v$, each endpoint of $P_v$ is adjacent to $v$.

If $v$ has only one neighbor in $P$, then $P_v$ has only one vertex and we are done. 
So we may assume that $v$ has at least two neighbors on $P$.
Suppose there are $u,w\in V(P_v)\cap N(v)$ such that $Q\coloneqq uP_vw$ has at least $q-1$ vertices and no internal vertex of $Q$ is a neighbor of $v$.
Then $vuQwv$ is an induced cycle of length at least $q$ in $G$, which is a contradiction.
Hence, $P_v$ consists of at most $2d-3$ segments of length at most $q-3$ each, where no internal vertex is adjacent to $v$.
It follows that $P_v$ has at most $(2d-3)(q-3)+1\le 2(d-1)(q-2)$ vertices.
\end{proof}

In the next step we show that the observation from \cref{lemma:pathinterval1} may be extended to entire components of $G-N[V(P)]$ if $G$ does not contain long induced cycles.

\begin{lemma}\label{lemma:componentattachment}
Let $d\ge 1$ and $q\ge 3$ be integers, let $G$ be a connected $K_{1,d}$-free graph without induced cycles of length at least $q$, and let $P$ be an induced path in $G$.
Let $H$ be a component of $G-N[V(P)]$ and let $v\in N(V(H))\cap N(V(P))$.
Then there exists a path $P_H\subseteq P$ on at most $2(d-1)(q-2)+2q$ vertices such that $N(V(H))\subseteq N(V(P_H))$.
\end{lemma}

\begin{proof}
By \cref{lemma:pathinterval1}, there exists a nonempty path $P_v\subseteq P$ on at most $2(d-1)(q-2)$ vertices such that $N(v)\cap V(P)\subseteq V(P_v)$ and each endpoint of $P_v$ is adjacent to $v$.
Let $p_1$ and $p_2$ be the two endpoints of $P_v$ (possibly $p_1 = p_2)$ and let $P_1'$ and $P_2'$ be two subpaths of $P$ such that the paths  $P_1'$, $P_v$, $P_2'$ are pairwise edge-disjoint and with union $P$, and $p_i\in V(P_i')$ for $i\in \{1,2\}$.

For each $i\in[2]$, if $P_i'$ has at most $q$ vertices let $P^i\coloneqq P_i'$, otherwise let $P^i$ be the subpath of $P_i'$ on $q+1$ vertices that contains $p_i$.
Let $P_H\coloneqq P^1p_1P_vp_2P^2$.
Notice that $P_H$ has at most $2(d-1)(q-2)+2q$ vertices.

It remains to show that $N(V(H))\subseteq N(V(P_H))$.
Suppose, for a contradiction, that there exists a vertex $u\in N(V(H))\setminus N(V(P_H))$.
Note that the vertex $u$, having a neighbor in $H$, cannot belong to $P_H$, hence $u\in N(V(H))\setminus N[V(P_H)]$.
Furthermore, since $H$ is a component of $G-N[V(P)]$ and $G$ is connected, $N(V(H))\subseteq N(V(P))$ and hence $u$ has a neighbor on $P$.
Since $u$ has no neighbors on $P_H$, we may assume without loss of generality that $u$ has a neighbor on $P_1'$.
Let $w$ be the neighbor of $u$ on $P_1'$ closest to $p_1$ along $P$.
The path $P^1$ consists of $q+1$ vertices, exactly one of which, namely $p_1$, is adjacent to $v$.
Let $R$ be a shortest $(u,v)$-path in $G[V(H)\cup\{u,v\}]$.
Then $vRuwP_1'p_1v$ is an induced cycle in $G$ that contains $P^1$ and thus has at least $q$ vertices, a contradiction.
\end{proof}

The previous results imply that, in the absence of long induced cycles, if a long induced path $P$ exists within $G$, then $G$ may be decomposed in a path-like fashion following the structure induced by the separator properties of the closed neighborhoods of the subpaths of $P$.
The next step is to start formalizing this intuition by building a path decomposition of $G[N[V(P)]]$ with bounded independence number, given some induced path $P$ as input.
The following two lemmas provide the last remaining tools to prove the key result of this section.

Let $G$ be a graph, $h\geq 1$ be an integer, and $P=v_1\dots v_{\ell}$ be an induced path on $\ell \geq h$ vertices in $G$.
Let $n\coloneqq \ell-h+1$ and let $B=b_1b_2\dots b_n$ be a path on $n$ vertices.
For each $i\in[n]$ let $P^i\coloneqq v_iv_{i+1}\dots v_{i+h-1}$ and set $\beta(b_i)\coloneqq N_G[V(P^i)]$.
We call the pair $(B,\beta)$ the \emph{$h$-backbone structure} of $P$ in $G$.

\begin{lemma}\label{lemma:backbone}
Let $d \geq 2$ and $q\geq 3$ be integers. 
Let $G$ be a connected $K_{1,d}$-free graph without induced cycles of length at least $q$. 
Moreover, let $h\geq q-1$ be an integer, let $P$ be an induced path on at least $h$ vertices in $G$, and let $(B,\beta)$ be the $h$-backbone structure of $P$ in $G$.
Then $(B,\beta)$ is a path decomposition of $G[N[V(P)]]$ with independence number at most $(d-1)h$.
\end{lemma}

\begin{proof}
Let $P=v_1\dots v_{\ell}$ and $B=b_1\dots b_n$.
For each $i\in[n]$, let $P^i$ be the subpath of $P$ from the construction of $(B,\beta)$.
We begin by showing that $(B,\beta)$ is indeed a path decomposition of $G'\coloneqq G[N[V(P)]]$.

To see this, first observe that for every $i\in[n]$ we have $v_i\in\beta(b_i)$ and for every $i\in[\ell]\setminus[n]$ we have $v_i\in\beta(b_n)$.
Moreover, if $v_i\in\beta(b_j)$, then $N(v_i)\subseteq\beta(b_j)$ as well.
Hence $\bigcup_{j\in[n]}\beta(b_j)=N[V(P)]$.

Next let $v_iv_{i+1}\in E(P)$. 
Then $v_i,v_{i+1}\in\beta(b_i)$ if $i\in[n]$ and $v_i,v_{i+1}\in\beta(b_n)$ otherwise.
A similar observation can be made for edges of the form $v_iu$ where $u\in N(V(P))$.
Now, let $u,w$ be two adjacent vertices from $N(V(P))$.
Then we may consider the subpaths $P_u$ and $P_w$ of $P$ from \cref{lemma:pathinterval1} together with the shortest $(V(P_u),V(P_w))$-subpath $Q$ of $P$.
In the case $Q$ contains at least $q-2$ vertices, $uQwu$ is an induced cycle of length at least $q$ in $G$, which is impossible.
Hence, $Q$ has at most $q-3$ vertices.
Let $v_i$ be the vertex of $Q$ minimizing $i$. 
Then, as $h\geq q-1$, it holds that $V(Q)\subseteq \beta(b_{\min\{ i,n\}})$.
Consequently $u,w\in \beta(b_{\min\{ i,n\}})$ as well, and thus every edge of $G'$ is contained in some bag of $(B,\beta)$.

Finally, suppose for a contradiction that there exists a vertex $u\in N[V(P)]$ together with $1\leq i<j<k\le n$ such that $u\in(\beta(b_i)\cap \beta(b_k))\setminus\beta(b_j)$.
Assume that $i$ and $k$ are chosen so that $k-i$ is minimal.
Notice that $u\in N(V(P))$ since the above situation is impossible for the vertices of $P$.
Let $p_i$ be the vertex of $P^i$ that is adjacent to $u$ and closest to $P^k$ on $P$.
Similarly, let $p_k$ be the vertex of $P^k$ that is adjacent to $u$ and closest to $P^i$ along $P$.
Clearly, $\{p_i,p_k\}\cap V(P^j)=\emptyset$ and thus the path $R\coloneqq p_iPp_k$, which contains $P_j$, has at least $h+2\ge q+1$ vertices.
Since $u$ has no neighbors in the interior of $R$, it follows that  $up_iPp_ku$ is an induced cycle of length at least $q$.  
This contradicts the assumption that $G$ has no induced cycle of length at least $q$.
Hence, $(B,\beta)$ is indeed a path decomposition of $G'$.

For every $b\in V(B)$, the fact that $\alpha(G'[\beta(b)])\le (d-1)h$ follows immediately from \cref{lemma:boundedalphapath}, and thus our proof is complete.
\end{proof}

\begin{lemma}\label{cor:attachments}
Let $d\geq 1$ and $q \geq 3$ be integers.
Let $G$ be a connected $K_{1,d}$-free graph without induced cycles of length at least $q$.
Moreover, let $h\geq 2(d-1)(q-2)+2q$ be an integer, let $P$ be an induced path on at least $h$ vertices in $G$, and let $(B,\beta)$ be the $h$-backbone structure of $P$ in $G$.
Then, for every component $H$ of $G-N[V(P)]$ there exists some $b_H\in V(B)$ such that $N(V(H))\subseteq \beta(b_H)$.
\end{lemma}

\begin{proof}
Let $B=b_1b_2\dots b_n$.
By \cref{lemma:componentattachment} there exists a path $P_H\subseteq P$ on at most $2(d-1)(q-2)+2q$ vertices such that $N(V(H))\subseteq N(V(P_H))$.
Moreover, since $h\geq 2(d-1)(q-2)+2q$, there exists some $i\in[n]$ such that the path $P^i$ used for the construction of $(B,\beta)$ contains $P_H$ as a subgraph.
Hence, $N(V(H))\subseteq \beta(b_i)$ and we can take $b_H = b_i$.
\end{proof}

We are now ready to prove the special case of \cref{thm:main} where $S$ and $T$ are connected.

\begin{theorem}\label{thm:excludelongclaw}
Let $d\ge 2$ and $p\ge 1$ be integers and let $G$ be a $\{K_{1,d},S_p,T_p\}$-free graph.
Then $\mathsf{tree}\text{-}\alpha(G)\leq 20(d-1)^4(p+1)$.
\end{theorem}

\begin{proof}
First assume $d=2$.
Then $G$ is $P_3$-free and thus $G$ is a disjoint union of complete graphs, which implies that $\mathsf{tree}\text{-}\alpha(G)\leq 1$. 
Hence, from now on we assume that $d\ge 3$.
We may also assume that $G$ is connected.

Let $q\coloneqq 2d(p+1)$ and $r\coloneqq 2(d-1)(q-2)$.
Notice that, in the case $G_0\coloneqq G$ does not contain an induced $P_{dp}$, we are done by \cref{thm:excludepathoriginal}.
Hence, there exists an induced path $P^0$ of $G_0$ on $dp$ vertices.
Let $X_0\coloneqq N_{G_0}[V(P^0)]$.

By \cref{lemma:apathskillsallcycles}, it follows that $G_1\coloneqq G_0-X_0$ does not have induced cycles of length at least~$q$.
Moreover, by \cref{lemma:boundedalphapath}, we know that $\alpha(G_0[X_0])\leq (d-1)dp$.
Our goal is to construct a tree decomposition of $G$ with independence number at most $20(d-1)^4(p+1)$.
It suffices to do this for the case when $G_1$ is connected.
In the case $G_1$ is not connected, the following arguments can be applied to each component of $G_1$ individually and the tree decompositions obtained may be joined by introducing an additional node $t$ whose bag consists exactly of the set $X_0$ and which is joined to exactly one node of each of the trees for the tree decompositions of the components.

In the case $G_1$ does not contain an induced $P_{6dq}$ we may, again, call upon \cref{thm:excludepathoriginal} to obtain a bound on $\tin(G_1)$.
By adding the vertices in $X_0$ to every bag of a tree decomposition of $G_1$ with minimum independence number (cf.~\cref{lem:an-easy-lemma}), we obtain a tree decomposition of $G$ with independence number at most 
\begin{align*}
 &\phantom{=~} 6(d-1)dq+(d-1)dp \\
 &= 12(d-1)d^2(p+1)+(d-1)dp \\
&\le (d-1)(p+1)\cdot d\cdot (12d+1) \\
 &\le (d-1)(p+1)\cdot 2(d-1)\cdot 10(d-1)^2=20(d-1)^4(p+1),
\end{align*}
for $G_0$ and are done.
So we may take $P^1=s_1 \dots s_m$ to be a longest induced path in $G_1$ and assume $m\geq 6dq$. 

Let $H$ be some component of $G_1-N_{G_1}[V(P^1)]$.
We claim that $H$ does not contain an induced path on $d(r+p-1)$ vertices.
Towards a contradiction let $F$ be an induced path on $d(r+p-1)$ vertices in $H$.
Moreover, let $Q=w_0w_1\dots w_{\ell}$ be a shortest $(V(P^1),V(F))$-path in $G_1$ with $w_0\in V(P^1)$.
Notice that $w_1\notin V(P^1)\cup V(F)$.
By \cref{lemma:pathinterval1} applied to $G_1$, $P^1$, and $w_1$, there exists a path $P_{w_1}\subseteq P^1$ with at most $r$ vertices such that $N_{G_1}(w_1)\cap V(P^1)\subseteq V(P_{w_1})$ and each endpoint of $P_{w_1}$ is adjacent to $w_1$.
Let $u$ be the neighbor of $w_1$ on $P^1$ that minimizes the distance to $s_1$ on $P^1$, and let $z$ be the neighbor of $w_1$ on $P^1$ that minimizes the distance to $s_m$ on $P^1$. See \cref{fig:theorem49} for an illustration.

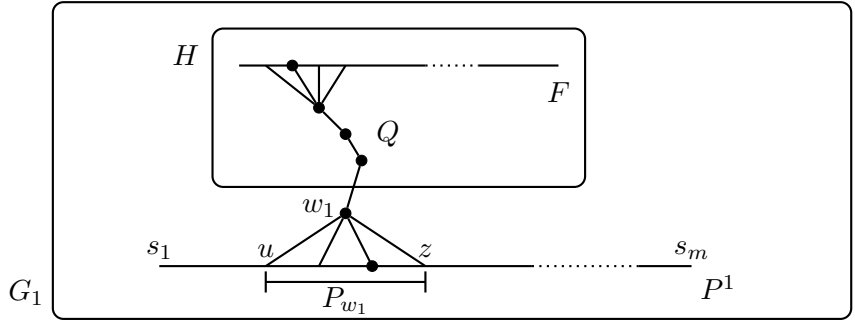
\begin{figure}
  \centering
  \begin{tikzpicture}[scale=0.7]
  \tikzstyle{w}=[circle,draw,fill=black,inner sep=0pt,minimum width=4pt]

\draw[rounded corners, thick] (0,1)--(0,0)--(15,0)--(15,6)--(0,6)--(0,1);
\draw[rounded corners, thick] (3,5)--(3,2.5)--(10,2.5)--(10,5.5)--(3,5.5)--(3,5);

\draw[rounded corners, thick] (2,1)--(9,1);
\draw[rounded corners, thick, dotted] (9,1)--(11,1);
\draw[rounded corners, thick] (11,1)--(12,1);

    \node at (12.5, 0.6) {$P^1$};
    \node at (2.5, 5) {$H$};
    \node at (9.5, 4.3) {$F$};
    \node at (6.3, 3.5) {$Q$};
     \node at (-0.5, 0.5) {$G_1$};
     \node at (5.5, 0.35) {$P_{w_1}$};
    \node at (2, 1.3) {$s_1$};
    \node at (4, 1.3) {$u$};
    \node at (7, 1.3) {$z$};
    \node at (5, 2.1) {$w_1$};
    \node at (12, 1.3) {$s_m$};

\draw[rounded corners, thick] (3.5,4.8)--(7,4.8);
\draw[rounded corners, thick, dotted] (7,4.8)--(8,4.8);
\draw[rounded corners, thick] (8,4.8)--(9.5,4.8);

  \draw (4.5, 4.8) node [w] (v1) {};
  \draw (5, 4) node [w] (v2) {};
  \draw (5.5, 3.5) node [w] (v3) {};
  \draw (5.8, 3) node [w] (v4) {};
  \draw (5.5, 2) node [w] (v5) {};
  \draw (6, 1) node [w] (v6) {};

  \draw[thick] (v2)--(v3)--(v4)--(v5);
  \draw[thick] (v2)--(4,4.8);
  \draw[thick] (v2)--(4.5,4.8);
  \draw[thick] (v2)--(5,4.8);
  \draw[thick] (v2)--(5.5,4.8);
   \draw[thick] (v5)--(4,1);
   \draw[thick] (v5)--(5,1);
   \draw[thick] (v5)--(6,1);
   \draw[thick] (v5)--(7,1);

  \draw[thick] (4, 0.9)--(4,0.5);
  \draw[thick] (7, 0.9)--(7,0.5);
  \draw[thick] (4, 0.7)--(7,0.7);

  \end{tikzpicture}     \caption{The paths $P^1$, $F$, and $Q$ in \cref{thm:excludelongclaw}. }\label{fig:theorem49}
\end{figure}
Since $F$ has at least $d(r+p-1)$ vertices, by \cref{lem:longsegment},
there is a segment of $F$ with respect to $w_{\ell-1}$ that has at least $r+p-2$ vertices that are not adjacent to $w_{\ell-1}$.
Thus, there exists an induced path $J$ on $r+p$ vertices in the graph $G_1[V(F)\cup V(Q)]-w_0$ that has $w_1$ as one endpoint and does not contain any other neighbor of~$P^1$. 
We distinguish two cases depending on neighbors of $w_1$ on $P^1$.

\paragraph{Case 1: $w_1$ has a neighbor in $s_{r+p+1}P^1s_{m-r-p}$.}

Since $P_{w_1}$ has at most $r$ vertices, $P^1-V(P_{w_1})$ has two components, each with at least $p$ vertices.

Let $L$ be the subpath of $s_1P^1u$ on $p+1$ vertices that contains $u$, and let $R$ be the subpath of  $zP^1s_m$ on $p+1$ vertices that contains $z$.
In the case $L$ and $R$ are not adjacent, or share an endpoint (which may happen if $|V(P_{w_1})| = 1$), the graph $G_1[V(L)\cup V(R)\cup V(J)]$ contains an induced $S_p$.
If they are adjacent, then there is exactly one edge between them and this edge joins their respective neighbors of $w_1$.
Hence, in this case $G_1[V(L)\cup V(R)\cup V(J)]$ contains an induced $T_p$.
In both cases, we obtain a contradiction.

\paragraph{Case 2: All neighbors of $w_1$ on $P^1$ belong either to $s_1P^1s_{r+p}$ or to $s_{m-r-p+1}P^1s_m$.}

Without loss of generality, we may assume that $w_1$ has a neighbor in $s_1P^1s_{r+p}$. 
Note that $2(r+p)+r-2=3r+2p-2<6dq$. 
Since $P_{w_1}$ has at most $r$ vertices and  $P^1$ has at least $6dq$ vertices, no vertex of $s_{m-r-p+1}P^1s_m$ is adjacent to $w_1$.

Note that the path $zP^1s_m$ contains exactly one neighbor of $w_1$ and has at least $m-r-p+1$ vertices.
Let $y$ be the endpoint of $J$ other than $w_1$.
Observe that $J$ has exactly one neighbor on $zP^1s_m$ and this vertex, namely $z$, is adjacent to exactly $w_1$ on $J$.
Hence, $yJw_1zP^1s_m$ is an induced path on at least $(m-r-p+1)+(r+p)>m$ vertices.
This is a contradiction to our assumption that $P^1$ is a longest induced path in $G_1$.

\medskip
Hence, our claim follows, that is, $H$ does not contain an induced path on $d(r+p-1)$ vertices.

\medskip
Now let $h \coloneqq 2dq$ and let $(B,\beta'')$ be the $h$-backbone structure of $P^1$ in $G_1$, where $B=b_1b_2\dots b_n$.
Since $h = 2dq\ge q-1$ and $G_1$ has no induced cycles of length at least $q$, \cref{lemma:backbone} implies that $(B,\beta'')$ is a path decomposition of $G_1[N_{G_1}[V(P^1)]]$ with independence number at most $(d-1)h=2(d-1)dq$.

Let $H$ be some component of $G_1-N_{G_1}[V(P^1)]$.
Since $h=2dq\ge 2(d-1)(q-2)+2q$, by \cref{cor:attachments} there exists a smallest $i_H\in[n]$ such that $N_{G_1}(V(H))\subseteq\beta''(b_{i_H})$.

It follows from the discussion above that $H$ excludes the path on $d(r+p-1)$ vertices as an induced subgraph.
Hence, by applying \cref{thm:excludepathoriginal} to each $H$, we may obtain a tree decomposition $(T_H,\beta_H)$ for $H$ with independence number at most $(d-1)d(r+p-3)$.
We combine these decompositions with $(B,\beta'')$ to form a tree decomposition $(T,\beta')$ of $G_1$ as follows.
For every $b\in V(B)$ set $\beta'(b)\coloneqq \beta''(b)$.
For every $H$ and $t\in V(T_H)$ set $\beta'(t)\coloneqq \beta_H(t)\cup\beta''(b_{i_H})$.
Then let $T$ be the tree obtained from the disjoint union of $B$ and all of the trees $T_H$ by joining, for every $H$, a single node of $T_H$ to the node $b_{i_H}$.
Observe that the resulting tuple $(T,\beta')$ is indeed a tree decomposition of $G_1$.
Moreover, for all $t\in V(T)$ we have 
\begin{align*}
    \alpha(G_1[\beta'(t)])& \leq (d-1)d(r+p-3)+2(d-1)dq \\
    &= (d-1)d(2(d-1)(q-2)+p-3+2q) \\
    &\le (d-1)d(2dq+p).
\end{align*}

Finally, to obtain a tree decomposition of $G$, we set $\beta(t)\coloneqq \beta'(t)\cup X_0$ for all $t\in V(T)$.
The resulting tree decomposition $(T,\beta)$ is now a tree decomposition of $G$ with independence number at most
\begin{align*}
 &\phantom{=~} (d-1)d(2dq+p)+(d-1)dp \\
 &= 2(d-1)d(dq+p) \\
 &\le 2(d-1)(p+1)\cdot d \cdot(2d^2+1) \\
 &\le 2(d-1)(p+1)\cdot 2(d-1)\cdot 5(d-1)^2=20(d-1)^4(p+1),
\end{align*}
as claimed.
\end{proof}

The general case, where $S$ and $T$ are not necessarily connected, follows from \cref{thm:excludelongclaw} via a straightforward induction.
More precisely, to prove \cref{thm:main} it suffices to observe the following.

\begin{corollary}\label{cor:main}
Let $d$, $k$, and $p$ be positive integers and let $G$ be a $\{K_{1,d},kS_p, kT_p\}$-free graph.
Then $\tin(G)< 6dk(p+1)+20d^4(p+1)$.
\end{corollary}

\begin{proof}
If $d = 1$, then $G$ is edgeless and $\tin(G)\le 1$, hence, the inequality holds.
So we may assume that $d\ge 2$.

We proceed by proving the following inequality by induction on $k$:
\begin{align*}
    \tin(G) \leq 6(d-1)(k-1)(p+1)+20(d-1)^4(p+1)
\end{align*}
Notice that, since $6(d-1)(k-1)(p+1)+20(d-1)^4(p+1)< 6dk(p+1)+20d^4(p+1)$, this will imply the assertion.
The case $k=1$ is handled by \cref{thm:excludelongclaw}, so we may immediately proceed with the inductive step for $k\geq 2$.

If $G$ contains an induced $S_p$, then let $X_0$ be the vertex set of an arbitrary induced $S_p$ in $G$ and $X$ be the closed neighborhood of $X_0$ in $G$.
If no induced $S_p$ exists in $G$, then let $X\coloneqq \emptyset$.
Notice that $\alpha(G[X])\leq |V(S_p)|(d-1)\leq 3(d-1)(p+1)$.

Similarly, if $G$ contains an induced $T_p$, we select a set $Y\subseteq V(G)$ such that $Y$ is the closed neighborhood of some induced $T_p$ in $G$.
Otherwise we set $Y\coloneqq \emptyset$.
As before we have $\alpha(G[Y])\leq |V(T_p)|(d-1)\leq 3(d-1)(p+1)$.

Now we may observe two things.
First, $\alpha(G[X\cup Y])\leq 6(d-1)(p+1)$.
Second, if the graph $G-(X\cup Y)$ contains an induced $(k-1) S_p$ or an induced $(k-1) T_p$, then $G$ contains an induced $k S_p$ or an induced $k T_p$, respectively, a contradiction.
Thus, $G-(X\cup Y)$ is $\{K_{1,d}, (k-1) S_p, (k-1) T_p\}$-free.
By our induction hypothesis, it follows that $\tin(G-(X\cup Y))\leq 6(d-1)(k-2)(p+1)+20(d-1)^4(p+1)$.
With the bound the independence number of $G[X\cup Y]$ and \cref{lem:an-easy-lemma} we obtain
\begin{align*}
    \tin(G) & \leq \alpha(G[X\cup Y]) + \tin(G-(X\cup Y))\\
    & \leq 6(d-1)(p+1) + 6(d-1)(k-2)(p+1)+20(d-1)^4(p+1)\\
    & \leq 6(d-1)(k-1)(p+1)+20(d-1)^4(p+1)\,,
\end{align*}
as desired.
\end{proof}

\section{Tree-independence number of line graphs\label{sec:line}}

In this section, we show that \cref{conj:main} holds for subclasses of the class of line graphs and determine the exact values of tree-independence number of line graphs of complete graphs and complete bipartite graphs.

\subsection{\texorpdfstring{\cref{conj:main}}{Conjecture 1.1} for subclasses of the class of line graphs}

We start by recalling a result of Bodlaender, Gustedt, and Telle~\cite{DBLP:conf/soda/BodlaenderGT98}.
The \emph{clique cover number} of a graph $G$ is the minimum number of cliques with union $V(G)$.
The proof of~\cite[Lemma 2.4]{DBLP:conf/soda/BodlaenderGT98} shows the following.

\begin{theorem}\label{thm:intersection-representation}
Let $H$ be a graph, let $\{H_j\}_{j\in J}$ be a family of connected subgraphs of $H$, and let $G$ be the graph with vertex set $J$ in which two distinct vertices $i$ and $j$ are adjacent if and only if $H_i$ and $H_j$ have a vertex in common.
Then $G$ has a tree decomposition $\mathcal T = (T,\beta)$ such that for each $t\in V(T)$, the induced subgraph $G[\beta(t)]$ has clique cover number at most $\tw(H)+1$.
\end{theorem}

Since the independence number of any graph $G$ is a lower bound on its clique cover number, \cref{thm:intersection-representation} implies the following.

\begin{corollary}\label{cor:intersection-representation}
Let $H$ be a graph, let $\{H_j\}_{j\in J}$ be a family of connected subgraphs of $H$, and let $G$ be the graph with vertex set $J$ in which two distinct vertices $i$ and $j$ are adjacent if and only if $H_i$ and $H_j$ have a vertex in common.
Then $\tin(G)\le \tw(H)+1$.
\end{corollary}

\cref{cor:intersection-representation} implies the following inequality relating the treewidth of a graph and the tree-independence number of its line graph. 

\begin{theorem}\label{thm:line}
For every graph $G$, it holds that $\tin(L(G))\le \tw(G)+1$.
Moreover, the bound is sharp: for every integer $n\ge 3$, there exists a graph $G$ such that $\tw(G) = n$ and $\tin(L(G))= n+1$.
\end{theorem}

\begin{proof}
For any graph $G$, applying \cref{cor:intersection-representation} to the case when $J = E(G)$ and for each edge $e\in E(G)$, the graph $H_e$ is the subgraph of $G$ induced by the endpoints of $e$, implies that $\tin(L(G))\le \tw(G)+1$.

We show that the bound is sharp with the following construction.
For an integer $n\ge 3$, let $G_n$ be the graph obtained from a complete graph of order $n$ by replacing each of its edges with two paths of length two joining the endpoints of the edge.
Thus, $G_n$ has $n+2{n\choose 2}$ vertices.
See \cref{fig:L(G_3)} for an illustration of $G_3$ and the remaining steps of this proof.
Let us denote by $V_n$ the set of $n$ vertices of the initial complete graph.
We fix an edge-coloring for $G_n$ with the colors red and blue such that for any two distinct vertices $u,v\in V_n$ the $4$-cycle composed by the two paths of length $2$ between $u$ and $v$ is properly edge-colored.
Note that every vertex of $V_n$ is incident with precisely $n-1$ red edges and $n-1$ blue edges.

Now let $H_n$ be the line graph of $G_n$.
For each vertex $v_i\in V_n$, the set of edges incident with $v_i$ correspond to a clique $C_i$ in $H_n$ with cardinality $2(n-1)$. 
Because of the above red-blue edge colorings, each of these cliques $C_i$ is partitioned into a ``red'' clique $R_i$ and a ``blue'' clique $B_i$, each with cardinality $n-1$.

\begin{figure}[ht]
    \centering
    \begin{tikzpicture}[scale=1]
        
        \pgfdeclarelayer{background}
		\pgfdeclarelayer{foreground}
			
		\pgfsetlayers{background,main,foreground}
			
			\begin{pgfonlayer}{main}
			
                \node (C) [v:ghost] {};
                \node (L) [v:ghost,position=180:50mm from C] {};
                \node (R) [v:ghost,position=0:50mm from C] {};

                \node (Llabel) [v:ghost,position=270:19mm from L] {(a)};
                \node (Clabel) [v:ghost,position=270:19mm from C] {(b)};
                \node (Rlabel) [v:ghost,position=270:19mm from R] {(c)};

                \node (X) [v:ghost,position=0:0mm from L] {

                    \begin{tikzpicture}[scale=1]
        
                        \pgfdeclarelayer{background}
		                  \pgfdeclarelayer{foreground}
			
		                  \pgfsetlayers{background,main,foreground}
			
			            \begin{pgfonlayer}{main}
			
                            \node (C) [v:ghost] {};

                            \node (u_1) [v:main,position=90:9mm from C] {};
                            \node (u_2) [v:main,position=210:9mm from C] {};
                            \node (u_3) [v:main,position=330:9mm from C] {};

                            \node (x_1) [v:main,position=150:12mm from C] {};
                            \node (x_2) [v:main,position=150:6mm from C] {};

                            \node (y_1) [v:main,position=270:12mm from C] {};
                            \node (y_2) [v:main,position=270:6mm from C] {};

                            \node (z_1) [v:main,position=30:12mm from C] {};
                            \node (z_2) [v:main,position=30:6mm from C] {};
   
                        \end{pgfonlayer}
			
			            \begin{pgfonlayer}{background}

                            \draw [e:main,line width=1.5pt,color=BostonUniversityRed] (u_1) to (x_1);
                            \draw [e:main,line width=1.5pt,color=CornflowerBlue] (u_1) to (x_2);
                            \draw [e:main,line width=1.5pt,color=CornflowerBlue] (u_1) to (z_1);
                            \draw [e:main,line width=1.5pt,color=BostonUniversityRed] (u_1) to (z_2);

                            \draw [e:main,line width=1.5pt,color=CornflowerBlue] (u_2) to (x_1);
                            \draw [e:main,line width=1.5pt,color=BostonUniversityRed] (u_2) to (x_2);
                            \draw [e:main,line width=1.5pt,color=BostonUniversityRed] (u_2) to (y_1);
                            \draw [e:main,line width=1.5pt,color=CornflowerBlue] (u_2) to (y_2);

                            \draw [e:main,line width=1.5pt,color=CornflowerBlue] (u_3) to (y_1);
                            \draw [e:main,line width=1.5pt,color=BostonUniversityRed] (u_3) to (y_2);
                            \draw [e:main,line width=1.5pt,color=BostonUniversityRed] (u_3) to (z_1);
                            \draw [e:main,line width=1.5pt,color=CornflowerBlue] (u_3) to (z_2);
   
			            \end{pgfonlayer}
   
			            \begin{pgfonlayer}{foreground}

			            \end{pgfonlayer}
   
                    \end{tikzpicture}};

                    \node (Y) [v:ghost,position=0:0mm from C] {

                    \begin{tikzpicture}[scale=1]
        
                        \pgfdeclarelayer{background}
		                  \pgfdeclarelayer{foreground}
			
		                  \pgfsetlayers{background,main,foreground}
			
			            \begin{pgfonlayer}{main}
			
                            \node (C) [v:ghost] {};

                            \node (A) [v:ghost,position=90:11mm from C] {};
                            \node (BB) [v:ghost,position=210:11mm from C] {};
                            \node (CC) [v:ghost,position=330:11mm from C] {};
                            \node (B) [v:ghost,position=180:2mm from BB] {};
                            \node (C) [v:ghost,position=0:2mm from CC] {};

                            \node (a_2) [v:main,position=180:2.5mm from A,color=CornflowerBlue] {};
                            \node (a_1) [v:main,position=180:6mm from a_2,color=BostonUniversityRed] {};
                            \node (a_3) [v:main,position=0:2.5mm from A,color=BostonUniversityRed] {};
                            \node (a_4) [v:main,position=0:6mm from a_3,color=CornflowerBlue] {};

                            \node (b_2) [v:main,position=135:2.5mm from B,color=BostonUniversityRed] {};
                            \node (b_1) [v:main,position=135:5mm from b_2,color=CornflowerBlue] {};
                            \node (b_3) [v:main,position=315:2.5mm from B,color=CornflowerBlue] {};
                            \node (b_4) [v:main,position=315:5mm from b_3,color=BostonUniversityRed] {};

                            \node (c_2) [v:main,position=45:2.5mm from C,color=CornflowerBlue] {};
                            \node (c_1) [v:main,position=45:5mm from c_2,color=BostonUniversityRed] {};
                            \node (c_3) [v:main,position=225:2.5mm from C,color=BostonUniversityRed] {};
                            \node (c_4) [v:main,position=225:5mm from c_3,color=CornflowerBlue] {};
   
                        \end{pgfonlayer}
			
			            \begin{pgfonlayer}{background}

                        \node (AAA) [v:ghost,position=90:1.2mm from A] {};
                        \node (BBB) [v:ghost,position=210:1.2mm from B] {};
                        \node (CCC) [v:ghost,position=330:1.2mm from C] {};

                        \node [ellipse, draw=black,minimum width=25mm,minimum height=9mm,opacity=1,rotate=0,position=0:0mm from AAA] {};

                        \node [ellipse, draw=black,minimum width=25mm,minimum height=9mm,opacity=1,rotate=135,position=0:0mm from BBB] {};

                        \node [ellipse, draw=black,minimum width=25mm,minimum height=9mm,opacity=1,rotate=225,position=0:0mm from CCC] {};

                        \draw [e:main] (a_1) to (b_1);
                        \draw [e:main] (a_2) to (b_2);

                        \draw [e:main] (b_3) to (c_3);
                        \draw [e:main] (b_4) to (c_4);

                        \draw [e:main] (c_1) to (a_4);
                        \draw [e:main] (c_2) to (a_3);

                        \draw [e:main,line width=1.5pt,bend left=55,color=ChromeYellow,densely dashed] (a_1) to (a_3);
                        \draw [e:main,line width=1.5pt,bend left=55,color=ChromeYellow,densely dashed] (a_2) to (a_4);

                        \draw [e:main,line width=1.5pt,bend right=55,color=ChromeYellow,densely dashed] (b_1) to (b_3);
                        \draw [e:main,line width=1.5pt,bend right=55,color=ChromeYellow,densely dashed] (b_2) to (b_4);

                        \draw [e:main,line width=1.5pt,bend left=55,color=ChromeYellow,densely dashed] (c_1) to (c_3);
                        \draw [e:main,line width=1.5pt,bend left=55,color=ChromeYellow,densely dashed] (c_2) to (c_4);
   
			            \end{pgfonlayer}
   
			            \begin{pgfonlayer}{foreground}

			            \end{pgfonlayer}
   
                    \end{tikzpicture}};

                    \node (Z) [v:ghost,position=0:0mm from R] {

                    \begin{tikzpicture}[scale=1]
        
                        \pgfdeclarelayer{background}
		                  \pgfdeclarelayer{foreground}
			
		                  \pgfsetlayers{background,main,foreground}
			
			            \begin{pgfonlayer}{main}
			
                            \node (Center) [v:ghost] {};

                            \node (A) [v:ghost,position=90:11mm from Center] {};
                            \node (BB) [v:ghost,position=210:11mm from Center] {};
                            \node (CC) [v:ghost,position=330:11mm from Center] {};
                            \node (B) [v:ghost,position=180:2mm from BB] {};
                            \node (C) [v:ghost,position=0:2mm from CC] {};

                            \node (a_1) [v:main,color=BostonUniversityRed,position=180:4mm from A] {};
                            \node (a_2) [v:main,color=CornflowerBlue,position=0:4mm from A] {};

                            \node (b_1) [v:main,color=CornflowerBlue,position=135:3.5mm from B] {};
                            \node (b_2) [v:main,color=BostonUniversityRed,position=315:3.5mm from B] {};

                            \node (c_1) [v:main,color=BostonUniversityRed,position=45:3.5mm from C] {};
                            \node (c_2) [v:main,color=CornflowerBlue,position=225:3.5mm from C] {};
   
                        \end{pgfonlayer}
			
			            \begin{pgfonlayer}{background}

                            \draw [e:main] (a_1) to (a_2);
                            \draw [e:main] (a_1) to (b_1);
                            \draw [e:main] (b_1) to (b_2);
                            \draw [e:main] (b_2) to (c_2);
                            \draw [e:main] (c_2) to (c_1);
                            \draw [e:main] (c_1) to (a_2);

                            \draw [e:main] (a_1) to (c_2);
                            \draw [e:main] (b_2) to (a_2);
                            \draw [e:main] (c_1) to (b_1);
   
			            \end{pgfonlayer}
   
			            \begin{pgfonlayer}{foreground}

			            \end{pgfonlayer}
   
                    \end{tikzpicture}};
   
            \end{pgfonlayer}
			
			\begin{pgfonlayer}{background}
   
			\end{pgfonlayer}
   
			\begin{pgfonlayer}{foreground}

			\end{pgfonlayer}
   
    \end{tikzpicture}
    \caption{An illustration for the proof of \cref{thm:line}: (a) shows the graph $G_3$ together with the edge coloring, (b) is the line graph $H_3$ of $G_3$ with the vertex coloring induced by the edge coloring of $G_3$, and (c) shows the resulting $K_{3,3}$ after contracting the dashed edges in (b).}
    \label{fig:L(G_3)}
\end{figure}
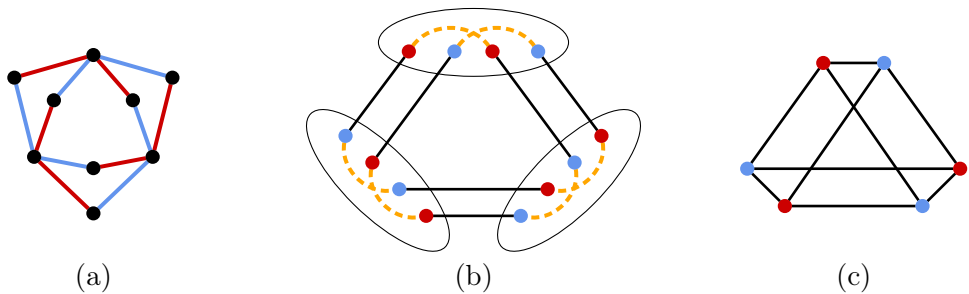

Note that if $i\neq j$, then there is no edge in $H_n$ between $R_i$ and $R_j$ and also no edge between $B_i$ and $B_j$.
Furthermore, for any $i,j$, there is an edge between $B_i$ and $R_j$.
Thus, contracting the edges within each of the monochromatic cliques $R_i$ and $B_i$ yields a complete bipartite graph $K_{n,n}$. 
Hence, $H_n$ contains $K_{n,n}$ as an induced minor.
By \cref{lem:tree-independence number induced-minor,tin-of-Knn}, we infer that $\tin(H_n) \ge n$.

For $n\ge 3$, turning $V_n$ into a clique transforms $G_n$ into a chordal graph with clique number $n$. 
Hence, $\tw(G_n)\le n-1$ by \cref{thm:tw-via-chordal-graphs}.

We conclude that $n\le \tin(L(G_n))\le \tw(G_n)+1\le n$ and hence, equalities hold.
\end{proof}

We remark that the statement of \cref{thm:line} is similar to the following.

\begin{theorem}[Dallard et al., Theorem 3.8 in~\cite{dallard2022firstpaper}]\label{thm:tin-tw} For every graph $G$, it holds that $\tin(G)\le\tw(G)+1$, and this bound is sharp: for every integer $n\ge 2$, there exists a graph $G$ such that $\tw(G) = n$ and $\tin(G) = n+1$.
\end{theorem}

Moreover, we are not aware of any graph $G$ with at least one edge such that $\tin(L(G))<\tin(G)$.
Note that if the inequality $\tin(G) \le \tin(L(G))$ holds for all graphs with at least one edge, then this would relate the inequalities from \cref{thm:line,thm:tin-tw} in a stronger sense: the former would imply the latter.

An \emph{$(n\times m)$-grid} is the graph with vertex set~${[n] \times [m]}$ and edge set 
\[
    {\{ \{ (i,j), (i,j+1) \} \mid i \in [n], j \in [m-1] \} \cup \{ \{ (i,j), (i+1,j) \} \mid i \in [n-1], j \in [m] \}}.
\] 
The \emph{elementary $k$-wall} for~${k \geq 3}$, is obtained from the ${(k \times 2k)}$-grid~$G_{k,2k}$ by deleting every odd edge in every odd column and every even edge in every even column, and then deleting all degree-one vertices.
See \cref{fig:wall} for an example.

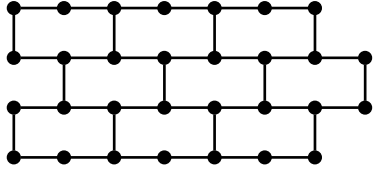
\begin{figure}
\centering
\begin{tikzpicture}[scale=1.1]
			\pgfdeclarelayer{background}
			\pgfdeclarelayer{foreground}
			\pgfsetlayers{background,main,foreground}

			\foreach \x in {1,...,7}
			\foreach \y in {1,4}
			{	
					\node[v:main,color=black] (v_\x_\y) at (\x*0.6,\y*0.6){};
			}

            \foreach \x in {1,...,8}
			\foreach \y in {2,3}
			{	
				\node[v:main,color=black] (v_\x_\y) at (\x*0.6,\y*0.6){};
			}
			
			\begin{pgfonlayer}{background}
			
				\foreach \y in {1,4}
				{
					\draw[e:main,color=black] (v_1_\y) to (v_7_\y);
				}

                \foreach \y in {2,3}
				{
					\draw[e:main,color=black] (v_1_\y) to (v_8_\y);
				}
			
				\foreach \x in {1,3,5,7}
				{
					\draw[e:main,color=black] (v_\x_1) to (v_\x_2);
					\draw[e:main,color=black] (v_\x_3) to (v_\x_4);
				}
			
				\foreach \x in {2,4,6,8}
				{
					\draw[e:main,color=black] (v_\x_2) to (v_\x_3);
				}
			
			\end{pgfonlayer}
		\end{tikzpicture}
  
    \caption{The elementary $4$-wall.}
    \label{fig:wall}
\end{figure}

As mentioned in the introduction, \cref{prop:line-graphs}, which we restate here for convenience, establishes \cref{conj:main} for subclasses of the class of line graphs.

\propLineGraphs*
\begin{proof} As remarked earlier, the implication $2 \Rightarrow 1$ follows from \cite{dallard2022firstpaper}. 
The implication $3 \Rightarrow 2$ follows from \cref{thm:line}. Finally, consider $1 \Rightarrow 3$. 
Let $\mathcal{G}$ be a class of graphs such that $L(\mathcal{G})$ is $(\tw,\omega)$-bounded. 
Suppose, to the contrary, that $\mathcal{G}$ has unbounded treewidth. 
By the Grid-Minor Theorem \cite{MR0854606}, there exists a function $f\colon \mathbb{N} \rightarrow \mathbb{N}$ such that, for each $k \in \mathbb{N}$, every graph of treewidth at least $f(k)$ contains a subdivision of the elementary $k$-wall as a subgraph. 
Let $G_{f(k)} \in \mathcal{G}$ be a graph of treewidth at least $f(k)$. 
Then, $G_{f(k)}$ contains a subdivision $W_k^{*}$ of the elementary $k$-wall $W_k$ as a subgraph.
This implies that $L(G_{f(k)})$ contains $L(W_k^{*})$ as an induced subgraph.
Note that $L(W_k)$ is a minor of $L(W_k^*)$ and hence $\tw(L(W_k^{*})) \geq \tw(L(W_k))$ by \cref{lem:tw-minor}.
Hence, $\tw(L(W_k^{*})) \geq \tw(L(W_k)) \geq (\tw(W_k)+1)/2 -1 \geq k/2 -1$, where the second inequality follows from~\cite{MR3820302}.
However, since $\omega(L(W_k^{*})) \leq 3$, we obtain a contradiction with the fact that $L(\mathcal{G})$ is $(\tw,\omega)$-bounded.
\end{proof}

A similar statement holds for the intersection graphs of connected subgraphs in some graphs.
For a class $\mathcal{G}$ of graphs, let $I(\mathcal{G})$ be the class of \emph{region intersection graphs} $G$ that can be obtained as follows.
For $H\in \mathcal{G}$ and a family $\{H_j\}_{j\in J}$ of connected subgraphs of $H$, let $G$ be the graph with vertex set $J$ in which two distinct vertices $i$ and $j$ are adjacent if and only if $H_i$ and $H_j$ have a vertex in common.

We remark that $L(\mathcal{G})$ is a subclass of $I(\mathcal{G})$ because to obtain the line graph of a graph, we can take $\{H_j\}_{j\in J}$ as the collection of all connected subgraphs with single edges.
Region intersection graphs have been studied as a common generalization of many classes of geometric intersection graphs (see~\cite{DBLP:conf/innovations/Lee17}).

\begin{theorem}
Let $\mathcal{G}$ be a class of graphs. 
Then, the following statements are equivalent.
\begin{enumerate}
\item The class $I(\mathcal{G})$ is $(\tw,\omega)$-bounded.
\item The class $I(\mathcal{G})$ has bounded tree-independence number.
\item The class $\mathcal{G}$ has bounded treewidth.
\end{enumerate}
\end{theorem}
\begin{proof}
    The implication $2 \Rightarrow 1$ follows from \cite{dallard2022firstpaper} and $3 \Rightarrow 2$ follows from \cref{cor:intersection-representation}. 
    To show $1 \Rightarrow 3$, suppose that $\mathcal{G}$ has unbounded treewidth.
    Then by~\cref{prop:line-graphs}, the class $L(\mathcal{G})$ of line graphs of graphs in $\mathcal{G}$ is not $(\tw, \omega)$-bounded. 
    Since $L(\mathcal{G})\subseteq I(\mathcal{G})$, the class $I(\mathcal{G})$ is also not $(\tw, \omega)$-bounded. 
\end{proof}

\subsection{Tree-independence number of \texorpdfstring{$L(K_{n,n})$ and $L(K_n)$}{L(K\_\{n,n\}) and L(K\_\{n\})}}

We now determine the exact values of the tree-independence number of line graphs of complete graphs and line graphs of complete bipartite graphs.
To put these results in perspective, recall that either exact or approximate values of sim-width and treewidth of these graphs are known and that any graph $G$ satisfies $\simw(G)\le \tin(G)\le \tw(G)+1$.

For treewidth, the exact values for $L(K_n)$ and $L(K_{n,n})$ (with $n\ge 3$) were settled by Harvey and Wood~\cite{HW15} and Lucena~\cite{Luc07}, respectively:
\[ \tw(L(K_{n})) = \left\lceil\frac{n^2}{4}+\frac{n}{2}-2\right\rceil,
\qquad
\tw(L(K_{n,n})) = \frac{n^2}2 + \frac n2 - 1\,.
\]
The latter result was extended by Harvey and Wood~\cite{MR3820302}, who showed that $\tw(L(K_{m,n}))$ has order $mn$.

For sim-width, Brettell et al.~\cite{brettell2023comparing} showed that $\simw(L(K_{m,n})) = \lceil m/3\rceil$, for any two integers $m$ and $n$ such that $6 < m \leq n$, and used this result to show that for all $n>12$, \[\left\lceil\frac{n}{6}\right\rceil\le \simw(L(K_n))\le \left\lceil\frac{2n}{3}\right\rceil\,.\]

In the next two results, we show that for line graphs of complete graphs and line graphs of complete bipartite graphs the upper bound on the tree-independence number given by the independence number is achieved with equality.

\begin{proposition}\label{biclique}
    For any two positive integers $m \leq n$, $\tin(L(K_{m,n})) = \alpha(L(K_{m,m})) = m$.
\end{proposition}
\begin{proof}
Let $G = L(K_{m,n})$ for some positive integers $m \leq n$.
More precisely, let
$V(G)=\{v_{i,j}\mid 1\le i \leq m, 1 \leq j\le n\}$, and
$E(G)=\{v_{i,j} v_{k,\ell}\mid (i=k \text{ and } j\neq \ell) \text{ or } (i\neq k\text{ and } j= \ell)\}$.
Note that $\alpha(G) = m$, since every independent set in $G$ corresponds to a matching in $K_{m,n}$ and the maximum number of edges in a matching in $K_{m,n}$ is $m$.
Consequently, $\tin(G) \leq \alpha(G) = m$.

We show that $\tin(G) \geq m$ by induction on $m$.
If $m=1$, the statement holds trivially.
Suppose that $m \geq 2$.
Let $H$ be an induced subgraph of $G$ isomorphic to $L(K_{m,m})$.
Let $\mathcal{T} = (T,\beta)$ be an arbitrary tree decomposition of $H$.
Our goal is to show that $\alpha(\mathcal{T})\geq m$, which by \cref{lem:tree-independence number induced-minor} would imply that $\tin(G) \geq \tin(H) \geq m$.

By \cref{lem:closednbr}, there is a vertex $v\in V(H)$ and a node $t\in V(T)$ such that $N[v]\subseteq \beta(t)$. 
By symmetry, we may assume without loss of generality that $v = v_{m,m}$.
Note that $H - N[v]$ is isomorphic to $L(K_{m-1,m-1})$.
By the induction hypothesis, we get that $\tin(H-N[v])= \tin(L(K_{m-1,m-1})) \ge m-1$.
This implies that $\mathcal{T}$ has a bag containing $m-1$ pairwise nonadjacent vertices from $H-N[v]$.
Let $t' \in V(T)$ be the node of $T$ that is closest to $t$ among all the nodes whose bag contains a set $Z$ of $m-1$ pairwise nonadjacent vertices in $H-N[v]$.
By the symmetry properties of $L(K_{m-1,m-1})$, we may assume without loss of generality that $Z = \{v_{i,i}\mid 1\le i\le m-1\}$.
Assume that $v_{i,m}, v_{m,i} \in \beta(t')$ for some $i \in \{1,\dots,m-1\}$. 
Then the set $Z' = (Z \setminus \{v_{i,i}\}) \cup \{v_{i,m}, v_{m,i}\}$ forms an independent set of size $m$ in $H$ such that $Z'\subseteq \beta(t')$, and hence $\alpha(\mathcal{T})\geq m$.
Similarly, if $v_{m,m}\in \beta(t')$, then the set $Z' = Z \cup \{v_{m,m}\}$ is an independent set of size $m$ in $H$ such that $Z'\subseteq \beta(t')$, and hence $\alpha(\mathcal{T})\geq m$.
We may thus assume that for each $i \in \{1,\dots,m\}$, at least one of $v_{i,m}$ and $v_{m,i}$ does not belong to $\beta(t')$.
In particular, $t'\neq t$.

Let $t''$ be the neighbor of $t'$ on the unique $(t',t)$-path in $T$ (possibly $t'' = t$).
The definition of $t'$ implies that $Z\not\subseteq \beta(t'')$, that is, there exists some $i\in \{1,\ldots,m-1\}$ such that $v_{i,i}\not\in \beta(t'')$.
We already know that at least one of $v_{i,m}$ and $v_{m,i}$ does not belong to $\beta(t')$. 
We may assume by symmetry that $v_{i,m}\not\in \beta(t')$.
Since $v_{i,i}v_{i,m} \in E(H)$, there exists a bag of $\mathcal{T}$ that contains this edge.
As $v_{i,m}\in N_{H}[v_{m,m}]\subseteq \beta(t)$ but $v_{i,m}\not\in \beta(t')$, the vertex $v_{i,m}$ cannot belong to any bag corresponding to a node of the subtree of $T-t't''$ containing $t'$.
Consequently, every bag of $\mathcal{T}$ that contains the edge $v_{i,i}v_{i,m}$ corresponds to a node of the subtree of $T-t't''$ containing $t''$.
But since $v_{i,i}\in \beta(t')\setminus \beta(t'')$, this contradicts the fact that bags of $\mathcal{T}$ containing the vertex $v_{i,i}$ form a connected subtree of $T$.
It follows that  this last case is not possible and we conclude that $\alpha(\mathcal{T})\geq m$ and consequently that $\tin(G) \geq \tin(H) \geq m$.
\end{proof}

\begin{proposition}\label{prop:LKn}
    For every positive integer $n$, $\tin(L(K_n)) = \alpha(L(K_n)) = \lfloor \frac{n}{2}\rfloor$.
\end{proposition}

\begin{proof}
To prove that $\tin(L(K_n))\ge \lfloor \frac{n}{2}\rfloor$, it is enough to see that since $K_n$ contains $K_{\lfloor \frac{n}{2}\rfloor,\lfloor \frac{n}{2}\rfloor}$ as a subgraph, $L(K_n)$ contains $L(K_{\lfloor \frac{n}{2}\rfloor,\lfloor \frac{n}{2}\rfloor})$ as an induced subgraph.
By \cref{biclique}, the inequality follows.
To prove that $\tin(L(K_n))\le \lfloor \frac{n}{2}\rfloor$, we use the fact that every independent set in $L(K_n)$ corresponds exactly to a matching of $K_n$. 
Since the largest size of a matching in $K_n$ is $\lfloor \frac{n}{2}\rfloor$ and the independence number of any graph is an upper bound on its tree-independence number, the result follows.   
\end{proof}

Since sim-width of any graph is bounded from above by its tree-independence number, \cref{prop:LKn} leads to the following improvement of the aforementioned inequality $\simw(L(K_n))\le \left\lceil\frac{2n}{3}\right\rceil$ due to Brettell et al.~\cite{brettell2023comparing}.

\begin{corollary}\label{cor:LKn}
    For every positive integer $n$, $\simw(L(K_n)) \le \lfloor \frac{n}{2}\rfloor$.
\end{corollary}

\section{Tree-independence number of \texorpdfstring{$P_4$}{P₄}-free graphs\label{sec:p4-free}}

In this section, we discuss the special case of $P_4$-free graphs.
These graphs have been widely studied for their rich algorithmic properties.
This is mostly due to the fact that a graph is $P_4$-free if and only if it is a \emph{cograph} (see, e.g.,~\cite{MR619603}), where the class of cographs is defined as the smallest class of graphs containing the one-vertex graph that is closed under the disjoint union and join operations.
Furthermore, cographs are exactly the graphs of \textit{modular width} two~\cite{MR0807891,DBLP:conf/iwpec/GajarskyLO13}.
We show here that, while their tree independence number is unbounded, it equals the size of a largest induced $K_{d,d}$-subgraph for cographs, a number that can be computed in linear time.
\restateKdd*

We begin with the following simple observation.

\begin{observation}\label{disjoint-union}
Let $G$ be the disjoint union of graphs $G_1$ and $G_2$. Then $\tin(G) =  \max\{\tin(G_1),\tin(G_2)\}$.
\end{observation}

Given a graph $G$, we denote by $\ibn(G)$ the \emph{induced biclique number} of $G$, that is, the largest nonnegative integer $n$ such that $G$ contains an induced subgraph isomorphic to $K_{n,n}$.

\begin{lemma}\label{ibn-lower-bound}
The induced biclique number of a graph is a lower bound on its tree-independence number.
More precisely, every non-null graph $G$ satisfies $\tin(G)\ge \max\{\ibn(G),1\}$.
\end{lemma}

\begin{proof}
This follows immediately from the fact that the tree-independence number cannot increase upon vertex deletion and that $\tin(K_{n,n}) = n$ (see~\cref{lem:tree-independence number induced-minor,tin-of-Knn}).
\end{proof}

The next result characterizes the tree-independence number of $P_4$-free graphs.

\begin{proposition}\label{tin-of-P4-tree-graphs}
Let $G$ be a $P_4$-free graph.
Then $\tin(G)=\max\{\ibn(G),1\}$.
\end{proposition}

\begin{proof}
By \cref{ibn-lower-bound}, it suffices to show that every $P_4$-free graph $G$ satisfies $\tin(G)\le \max\{\ibn(G),1\}$.
We show this using strong induction on $n = |V(G)|$.

The case $n = 1$ is trivial: a tree decomposition with a single bag containing the unique vertex has tree-independence number $1$.

Let $G$ be a $P_4$-free graph with $n>1$ vertices and assume that, for every $P_4$-free graph $G'$ with fewer than $n$ vertices, it holds that $\tin(G')\le \max\{\ibn(G'),1\}$.
Since $G$ has $n > 1$ vertices, there exist two $P_4$-free graphs $G_1$ and $G_2$ such that $G$ can be obtained either from the disjoint union of $G_1$ and $G_2$, or from the join of $G_1$ and $G_2$.
Assume first that $G$ is the disjoint union of $G_1$ and $G_2$.
In particular, $G$ is disconnected.
By the induction hypothesis, we have $\tin(G_i)\le \max\{\ibn(G_i),1\}$ for $i= 1,2$.
By \cref{disjoint-union}, we have $\tin(G)=\max\{\tin(G_1),\tin(G_2)\}$.
Since also $\ibn(G) = \max\{\ibn(G_1),\ibn(G_2)\}$, we obtain 
\begin{align*}
\tin(G) = &\max\{\tin(G_1),\tin(G_2)\}\\
\le&\max\{\max\{\ibn(G_1),1\}, \max\{\ibn(G_2),1\}\}  \\
=&\max\{\ibn(G_1),\ibn(G_2),1\}= \max\{\ibn(G),1\}\,,
\end{align*}
as desired.

Assume now that $G$ is the join of $G_1$ and $G_2$.
Since $G$ contains an induced subgraph isomorphic to $K_{1,1}$, we have $\max\{\ibn(G),1\} = \ibn(G)$. 
By the induction hypothesis, we have $\tin(G_i)\le \max\{\ibn(G_i),1\}$ for $i= 1,2$.
Every induced subgraph of $G$ isomorphic to some $K_{p,p}$ for $p\ge 1$ is either fully contained in $G_i$ for some $i\in \{1,2\}$ or has one set of the bipartition in $G_1$ and the other one in $G_2$.
The former ones show that $\ibn(G)\ge \ibn(G_i)$ for $i\in \{1,2\}$, and the latter ones that $\ibn(G)\ge \min\{\alpha(G_1), \alpha(G_2)\}$.
More precisely, we have \[\ibn(G) = \max\{\ibn(G_1),\ibn(G_2),\min\{\alpha(G_1),\alpha(G_2)\}\}\,.\]
By symmetry, we may assume without loss of generality that $\alpha(G_1)\le\alpha(G_2)$.
By the induction hypothesis, we have $\tin(G_2)\le \max\{\ibn(G_2),1\}$, and hence there exists a tree decomposition of $G_2$ with independence number at most $\max\{\ibn(G_2),1\}$.
Adding the vertices of $G_1$ to each bag of such a tree decomposition results in a tree decomposition of $G$ with independence number at most $\max\{\alpha(G_1),\max\{\ibn(G_2),1\}\}=\max\{\alpha(G_1),\ibn(G_2)\}$.
Therefore, \[\tin(G)\le \max\{\alpha(G_1),\ibn(G_2)\}
\le \max\{\ibn(G_1),\ibn(G_2),\alpha(G_1)\} = \ibn(G)\,,\]
which completes the proof.
\end{proof}

The recursive decomposition of a $P_4$-free graph into components of the graph or its complement all the way down to the copies of the one-vertex graph can be described using a decomposition tree called a \emph{cotree} and can be computed in linear time using modular decomposition, as shown by Corneil et al. \cite{MR0807891}.
Following the cotree from the leaves to the root yields a linear-time algorithm to compute the independence number of a $P_4$-free graph, using the recurrence relations $\alpha(G_1+G_2) = \alpha(G_1)+\alpha(G_2)$,
$\alpha(G_1\ast G_2) = \max\{\alpha(G_1),\alpha(G_2)\}$ (or, more precisely, their obvious generalizations to the disjoint unions and joins of any number of graphs) and the initial condition $\alpha(K_1) = 1$.
Consequently, the induced biclique number of a $P_4$-free graph can also be computed in linear time using the relations
\begin{align*}
&\ibn(G_1+G_2) = \max\{\ibn(G_1),\ibn(G_2)\}\\
&\ibn(G_1\ast G_2) = \max\{\ibn(G_1),\ibn(G_2),\min\{\alpha(G_1),\alpha(G_2)\}\}
\end{align*}
(or, more precisely, their generalizations to the disjoint unions and joins of any number $k\ge 2$ of graphs) and the initial condition $\ibn(K_1) = 0$.
Thus, \cref{tin-of-P4-tree-graphs} has the following consequence.

\begin{corollary}
The tree-independence number of a $P_4$-free graph can be computed in linear time.
\end{corollary}

\section{Conclusion}\label{sec:conc}

Towards a possible resolution of \cref{conj:main} for hereditary graph classes defined by a finite set of forbidden induced subgraphs, it suffices to prove \cref{conjecture:finitely-many-fis}.
In this paper we have made a first step towards \cref{conjecture:finitely-many-fis} by proving that the conjecture holds in the case where we replace $K_{d,d}$ with $K_{1,d}$.
This setting seems natural as it can be seen as an ``induced'' generalization of graphs of bounded maximum degree.

Korhonen showed in \cite{MR4539481} that, for graphs with bounded maximum degree, an induced variant of the Grid-Minor Theorem~\cite{MR0854606} holds for treewidth.

\begin{theorem}
There exists a function $f\colon\mathbb{N}^2\to\mathbb{N}$ such that for every positive integer $k$ and every graph $G$, if $\mathsf{tw}(G)>f(\Delta(G),k)$, then $G$ contains the $(k\times k)$-grid as an induced minor.
\end{theorem}

Graphs with bounded degree have bounded clique number, which implies that, in this setting, bounded treewidth, $(\tw,\omega)$-boundedness, and bounded tree-independence number are equivalent to each other.
Hence, in the context of tree-independence number, it seems natural to conjecture a generalization of the induced variant of the Grid-Minor Theorem for graphs with bounded degree to hereditary graph classes excluding some $K_{1,d}$.
Given a hereditary graph class $\mathcal{C}$, we call the class $\widehat{\mathcal{C}}\coloneqq \{ H \mid H\text{ is an induced minor of some }G\in\mathcal{C} \}$ the \emph{induced minor closure} of $\mathcal{C}$.

\begin{conjecture}\label{conj:inducedgrid}
Let $d$ be a positive integer and $\mathcal{C}$ be a hereditary graph class excluding $K_{1,d}$.
Then $\mathcal{C}$ has bounded tree-independence number if and only if $\widehat{\mathcal{C}}$ does not contain all planar graphs.
\end{conjecture}

It is a well known fact that every planar graph is a minor of a large enough wall \cite{MR0854606}.
Moreover, for any graph $G$ it holds that if $G$ has a graph $H$ as a minor, then the graph obtained from $G$ by subdividing every edge once contains $H$ as an induced minor \cite{DBLP:conf/innovations/Lee17}.
This implies that \cref{conj:inducedgrid} may be stated equivalently in terms of induced subgraphs.

\begin{conjecture}\label{conj:inducedgrid2}
Let $d$ be a positive integer and $\mathcal{C}$ be a hereditary graph class excluding $K_{1,d}$.
Then $\mathcal{C}$ has bounded tree-independence number if and only if there exists a positive integer $k$ such that $\mathcal{C}$ excludes all subdivisions of the elementary $k$-wall and their line graphs.
\end{conjecture}

Returning to our discussion about hereditary graph classes defined by excluding a finite set of forbidden graphs, recall that a special case of \cref{conjecture:finitely-many-fis} is \cref{conjecture:excluding-a-path}, stating that for any two positive integers $d$ and $s$, the class of $\{K_{d,d},P_s\}$-free graphs has bounded tree-independence number.
We observed that \cref{conjecture:excluding-a-path} holds for every $d$ when excluding $K_{d,d}$ and $P_4$, and improved an exponential upper bound on the tree-independence number of $\{K_{d,d},P_4\}$-free graphs that follows from results in the literature to a sharp linear upper bound, obtaining along the way a linear-time algorithm to compute the tree-independence number of a $P_4$-free graph.
In addition, \cref{thm:excludepathoriginal} proves \cref{conjecture:excluding-a-path} for every $d$ and $s$ when excluding $K_{1,d}$ and $P_s$.
A natural next step would be to approach the following further weakening of \cref{conjecture:excluding-a-path}.

\begin{conjecture}\label{conj:K2dPs}
For any two positive integers $d$ and $s$, the class of $\{ K_{2,d},P_s\}$-free graphs has bounded tree-independence number.
\end{conjecture}

\cref{tin-of-P4-tree-graphs-original} implies that \cref{conj:K2dPs} holds for $d = 1$ or $s\le 4$.
It also holds for $d = 2$ and $s = 5$. 
In this case we are dealing with $\{P_5,C_4\}$-free graphs and it can be shown that every such graph $G$ has tree-independence number at most~$2$.
Indeed, if $G$ is not chordal, then $G$ contains an induced $5$-cycle, and analyzing the possible ways in which the neighbors of a fixed $5$-cycle connect to the cycle, a structural characterization of $\{P_5,C_4\}$-free graphs can be obtained, which implies the existence of a tree decomposition with independence number $2$.

Finally, recall that we established the validity 
of \cref{conj:main} for subclasses of the class of line graphs.
The inequalities relating the treewidth of a graph with the tree-independence number of either the graph or its line graph (cf.\ \cref{thm:line,thm:tin-tw}) motivate the question of whether the tree-independence cannot decrease when taking the line graph.

\begin{question}
Is $\tin(L(G))\ge \tin(G)$ if $G$ is not edgeless?
\end{question}

\paragraph{Acknowledgements.}

This work is supported in part by the Slovenian Research and Innovation Agency (I0-0035, research programs P1-0285 and P1-0383, research projects J1-3001, J1-3002, J1-3003, J1-4008, J1-4084, and N1-0102), and by the research program CogniCom (0013103) at the University of Primorska, by the National Research Foundation of Korea (NRF) grant funded by the Ministry of Science and ICT (No. NRF-2021K2A9A2A11101617 and RS-2023-00211670), and the Institute for Basic Science (IBS-R029-C1).

\bibliographystyle{abbrv}
\bibliography{biblio}

\end{document}